\documentclass[11pt]{amsart}  
\usepackage{amsthm, amsmath, amscd, amssymb, latexsym, stmaryrd, color}

\usepackage[all]{xypic}
\input xypic
\usepackage{color}
\usepackage[top=3.0cm, bottom=3.0cm, left=3.0cm, right=3.0cm]{geometry}

\theoremstyle{plain}

\newtheorem{theoreme}{Theorem}
\newtheorem{coro}{Corollary}
\newtheorem{conjj}{Conjecture}
\newtheorem{theorem}{Theorem}[section]
\newtheorem{lemma}[theorem]{Lemma}
\newtheorem{proposition}[theorem]{Proposition}
\newtheorem{prop-def}[theorem]{Proposition-Definition}

\newtheorem{conjecture}[theorem]{Conjecture}

\theoremstyle{definition}

\theoremstyle{remark}
\newtheorem{remark}[theorem]{Remark}
\newtheorem*{ack}{Acknowledgement}

\usepackage[mathscr]{eucal}
\usepackage{graphics, graphpap}
\usepackage{array, tabularx, longtable}
\usepackage{color}
\numberwithin{equation}{section}

\def\Var{\mathrm{Var}}
\def\Mor{\mathrm{Mor}}

\def\loc{\mathrm{loc}}

\def\pr{\mathrm{pr}}

\def\sr{\mathrm{sr}}

\def\ord{\mathrm{ord}}

\def\Jac{\mathrm{Jac}}

\def\Spec{\mathrm{Spec}}

\def\Sing{\mathrm{Sing}}
\def\Id{\mathrm{Id}}

\def\ac{\mathrm{ac}}

\def\L{\mathbb{L}}

\title[Equivariant motivic integration and integral identity]{Equivariant motivic integration and proof of the integral identity conjecture for regular functions}
\author{L\^e Quy Thuong}
\date{}
\address{Department of Mathematics, Vietnam National University, Hanoi \newline
\indent 334 Nguyen Trai Street, Thanh Xuan District, Hanoi, Vietnam}
\email{leqthuong@gmail.com}

\thanks{The first author's research is funded by Vietnam National Foundation for Science and Technology Development (NAFOSTED) under grant number FWO.101.2015.02. The second author is supported by the National Foundation for Science and Technology Development (NAFOSTED), Grant number 101.04-2017.12, Vietnam}

\author{Nguyen Hong Duc}
\address{$^{\dag}$Hanoi Institute of Mathematics\newline \indent 18 Hoang Quoc Viet, Hanoi, Vietnam.} 
\email{nhduc82@gmail.com}
\address{$^{\dag}$Basque Center for Applied Mathematics, \newline \indent Alameda de Mazarredo 14, 48009 Bilbao, Bizkaia, Spain.} 
\email{hnguyen@bcamath.org}

\thanks{This research is also supported by ERCEA Consolidator Grant 615655 - NMST and by the Basque Government through the BERC 2014-2017 program and by Spanish Ministry of Economy and Competitiveness MINECO: BCAM Severo Ochoa excellence accreditation SEV-2013-0323.}


\begin{document}           

\begin{abstract}
We develop the Denef-Loeser motivic integration to the equivariant motivic integration and use it to prove the full integral identity conjecture for regular functions. 
\end{abstract}

\maketitle  

\section{Introduction}\label{sec1}
\subsection{}
Throughout this article, we shall work with a base field of characteristic zero $k$, as in the setup of the classical motivic integration in \cite{DL2}. Introduced by Kontsevich in 1995, motivic integration quickly becomes a crucial object in algebraic geometry on account of its connection to many fields of mathematics, such as mathematical physics, birational geometry, non-Archimedean geometry, tropical geometry, singularity theory, Hodge theory, model theory (cf. \cite{DL1}, \cite{DL2}), \cite{Ba}, \cite{Se}, \cite{LS}, \cite{HK}, \cite{NS}, \cite{Ni}, \cite{CL}, \cite{KS}). Let $X$ be an algebraic $k$-variety, and let $\mathscr L(X)$ be its arc space (cf. Section \ref{Prel}). According to \cite{DL2}, one can construct an additive measure $\mu$ on the family of {\it semi-algebraic} subsets of $\mathscr L(X)$ which takes value in a completion of the Grothendieck ring of algebraic $k$-varieties $\mathscr M_k$. There is a subfamily $\mathscr F_X$ consisting of {\it stable} semi-algebraic subsets of $\mathscr L(X)$ whose measure is exactly in $\mathscr M_k$; the restriction of $\mu$ on $\mathscr F_X$ will be denoted by $\tilde\mu$. The motivic integral of a {\it simple} function $\ell: A\to \mathbb Z$ on a semi-algebraic subset $A$ of $\mathscr L(X)$ can be defined if all the fibers of $\ell$ are semi-algebraic (i.e., measurable) and if the completion of $\mathscr M_k$ is ``big enough'' so that the sum $\sum_{n\in \mathbb Z}\mu(\ell^{-1}(n))\L^{-n}$ makes sense. Here, $\L$ is the class of $\mathbb A_k^1$ in the Grothendieck ring, which is invertible in $\mathscr M_k$. If for such a simple function $\ell$ all of its fibers and $A$ are in $\mathscr F_X$, then by a result in \cite{DL2}, $|\ell|$ is bounded and the sum 
\begin{align}\label{integration}
\int_A\L^{-\ell}d\tilde\mu:=\sum_{n\in \mathbb Z}\tilde\mu(\ell^{-1}(n))\L^{-n}
\end{align}
is in $\mathscr M_k$. A breakthrough point of motivic integration is that it admits a change of variables formula with respect to a proper birational morphism onto $X$, which Kontsevich used to prove Batyrev's conjecture on the Betti number of birationally equivalent complex Calabi-Yau varieties (cf. \cite{Ba}).

\subsection{}\label{sec1.2}
We shall concern an algebraic group $G$ and its good actions on $k$-schemes. The context where we investigate will be the $G$-equivariant Grothendieck ring $K_0^G(\Var_S)$ of $S$-varieties endowed with good $G$-action, and its localization $\mathscr M_S^G:=K_0^G(\Var_S)[\L^{-1}]$, where $S$ is an algebraic $k$-variety with trivial $G$-action, and $\L$ is the $S$-isomorphism class of $\mathbb A_S^1$ with affine $G$-action (cf. Section \ref{Prel}). The notion of piecewise trivial fibration for maps between constructible sets introduced in \cite{DL2} can be easily generalized to that of $G$-equivariant piecewise trivial fibration. The first fundamental result of the present article is the following theorem.

\begin{theoreme}[Theorem \ref{fund2}] \label{Thm1}
Let $X$, $Y$ and $F$ be algebraic $k$-varieties endowed with a good $G$-action, and let $f\colon X\to Y$ be a $G$-equivariant morphism. Assume that the categorical quotient $Y\to Y/\!/G$ exists and it is quasi-finite (e.g., $G$ is a finite group). Then $f$ is a $G$-equivariant piecewise trivial fibration if and only if for every $y$ in $Y$, there exists an $G_y$-equivariant isomorphism $X_y \stackrel{\cong}{\longrightarrow} F\times_k k(y)$.
\end{theoreme}

According to this theorem, if one of two equivalent conditions is satisfied, then we get the identity $[X]=[Y]\cdot [F]$ in $\mathscr M_k^G$. In the special case where $F$ is the affine variety $\mathbb A_k^n$, the following theorem is much more applicable than the previous one.

\begin{theoreme}[Theorem \ref{keycoro}]\label{keyco}
Let $X$ and $Y$ be algebraic $k$-varieties endowed with good $G$-action. Let $f\colon X\to Y$ be a $G$-equivariant morphism. Assume that the categorical quotient $Y\to Y/\!/G$ exists and it is quasi-finite (e.g., $G$ is a finite group). Assume moreover that for every $y$ in $Y$, there exists an isomorphism of $k(y)$-varieties $X_y\cong \mathbb{A}^n_{k(y)}$, for a given $n$ in $\mathbb N$.
Then the identity 
$$ [X]=[Y]\cdot \mathbb{L}^n$$
holds in $K_0^G(\Var_k)$.
\end{theoreme}

\subsection{}\label{sec1.3}
Let $\hat G$ be a group $k$-scheme of the form $\hat G=\varprojlim_{i\in I} G_i$, where $I$ is a partially ordered directed set and $\left\{G_i, G_j\to G_i \mid i\leq j \ \text{in}\ I\right\}$ is a projective system of algebraic groups over $k$. We shall consider good $\hat G$-actions on $k$-schemes ($\hat G$ acts on a scheme via some $G_i$) and develop the integral (\ref{integration}) to the $\hat G$-equivariant version. We define $K_0^{\hat G}(\Var_S):=\varinjlim  K_0^{G_i}(\Var_S)$ and $\mathscr M_S^{\hat G}:=K_0^{\hat G}(\Var_S)[\L^{-1}]$, and we have $\mathscr M_S^{\hat G}=\varinjlim  \mathscr M_S^{G_i}$. 

Assume that there exists a {\em nice} $\hat G$-action on $\mathscr L(X)$, i.e., for every $i$ in $I$ and $n$ in $\mathbb N$, there is a good $G_i$-action on $\mathscr L_n(X)$ such that the morphism $\pi^m_n: \mathscr L_m(X) \to \mathscr L_n(X)$ is $G_i$-equivariant for all $i\in I$ and $m\geq n$. For each $i$ in $I$, let $\mathscr F_X^{G_i}$ be the family of $A$ in $\mathscr F_X$ such that $\pi_m(A)$ is $G_i$-invariant for every $m\geq n$, with $n$ the stable level of $A$. Define $\mathscr F_X^{\hat G}:=\varprojlim\mathscr F_X^{G_i}$. Using Theorem \ref{keyco} (i.e., Theorem \ref{keycoro}), we can construct a well defined $\hat G$-equivariant additive measure 
\begin{align}\label{muG}
\tilde\mu^{\hat G}: \mathscr F_X^{\hat G} \to \mathscr M_k^{\hat G} 
\end{align}
and a natural $\hat G$-equivariant motivic integral 
$$\int_A\L^{-\ell}d\tilde\mu^{\hat G}:=\sum_{n\in \mathbb N}\tilde\mu^{\hat G}\left(\ell^{-1}(n)\right)\L^{-n},$$
which takes value in $\mathscr M_k^{\hat G}$, provided $A$ and all the fibers of a natural-value simple function $\ell$ are in $\mathscr F_X^{\hat G}$. Furthermore, we perform in Theorem \ref{changeofvariable} that the $\hat G$-equivariant motivic integration admits a change of variables formula with respect to a proper birational morphism onto $X$ (cf. Section \ref{EquivMotInt}).

\subsection{}
In significant applications, one takes $\hat G$ to be the profinite group scheme of roots of unity $\hat\mu$, the projective limit of the group schemes $\mu_n=\Spec \left(k[T]/(T^n-1)\right)$ and transition morphisms $\mu_{mn}\to \mu_n$ given by $\lambda\mapsto \lambda^m$. For this case, and for the natural $\hat\mu$-action on $\mathscr L(X)$ (cf. \cite{DL1}), we write simply $\tilde\mu$ instead of $\tilde\mu^{\hat\mu}$, the $\hat\mu$-equivariant measure on $\mathscr F_X^{\hat\mu}$, cf. (\ref{muG}).

Assume the algebraic $k$-variety $X$ is smooth of pure dimension $d$. We consider a regular function $f$ on $X$ whose zero locus $X_0$ is nonempty. For every $n\geq 1$, the sets 
\begin{gather*}
\overline{\mathscr X}_n(f):=\left\{\gamma\in \mathscr{L}(X) \mid f(\gamma)= t^n\mod t^{n+1}\right\}, \\ 
\overline{\mathscr X}_{n,x}(f):=\left\{\gamma\in \overline{\mathscr X}_n(f), \gamma(0)=x\right\},
\end{gather*}
with $x$ a closed point in $X_0$, are in $\mathscr F_X^{\mu_n}$ with stable level being $n$. Notice that the $k$-variety $\mathscr X_n(f):=\pi_n(\overline{\mathscr X}_n(f))$ admits the natural morphism to $X_0$ sending $\gamma$ to $\gamma(0)$. In view of \cite{DL1}, the motivic zeta functions $\sum_{n\geq 1}[\mathscr{X}_n(f)]\L^{-nd}T^n$ in $\mathscr M_{X_0}^{\hat{\mu}}[[T]]$ and $\L^d\sum_{n\geq 1}\tilde\mu(\overline{\mathscr X}_{n,x}(f))T^n$ in $\mathscr M_k^{\hat{\mu}}[[T]]$ are {\it rational} series. The {\it limit} of these series, $\mathscr S_f$ and $\mathscr S_{f,x}$, are called the motivic nearby cycles of $f$ and the motivic Milnor fiber of $f$ at $x$, respectively (cf. Section \ref{MZF}).  

We want to go further on the rationality of a formal power series with coefficients in $\mathscr M_k^{\hat\mu}$. The following theorem is the first attempt, it proves a special case of Conjecture \ref{conj2}. We do not need $X$ to be smooth in our result.

\begin{theoreme}[Theorem \ref{Thm3.4}]\label{Thm2}
Let $X$ be an algebraic $k$-variety and $f$ a regular function on $X$. Let $A_{\alpha}$, $\alpha$ in $\mathbb N^r$, be a semi-algebraic family of semi-algebraic subsets of $\mathscr L(X)$ such that, for every covering of $X$ by affine open subsets $U$, any semi-algebraic condition defining $A_{\alpha}\cap \mathscr L(U)$ contains only conditions of two first forms in (\ref{eqadd}). Assume that, for every $\alpha$ in $\mathbb N^r$, $A_{\alpha}$ is stable and disjoint with $\mathscr L(X_{\Sing})$. For $n\geq 1$, we put 
$$A_{n,\alpha}:=\left\{\gamma\in A_{\alpha} \mid f(\gamma)=t^n\mod t^{n+1}\right\},$$
which is in $\mathscr F_X^{\mu_n}$. Then the series $\sum_{(n,\alpha)\in \mathbb N^{r+1}}\tilde{\mu}\left(A_{n,\alpha}\right)T_0^nT_1^{\alpha_1}\cdots T_r^{\alpha_r}$ is rational.
\end{theoreme}

A consequence of this theorem plays an important role in the last section of the article.

\begin{coro} [Proposition \ref{Prop4.5}]
Let $X$, $f$, $A_{\alpha}$ and $A_{n,\alpha}$ be as in Theorem \ref{Thm2}. Let $\Delta$ be a cone in $\mathbb{R}^{r+1}_{\geq 0}$ and $\bar{\Delta}$ its closure. Let $\ell$ and $\varepsilon$ be integral linear forms on $\mathbb Z^{r+1}$ with $\ell(n,\alpha)>0$ and $\varepsilon(n,\alpha)\geq 0$ 
for all $(n,\alpha)$ in $\bar{\Delta}\setminus \{0\}$. Then the series $\sum_{(n,\alpha)\in \Delta\cap \mathbb N^{r+1}}\tilde{\mu}\left(A_{n,\alpha}\right)\L^{-\varepsilon(n,\alpha)}T^{\ell(n,\alpha)}$ is rational, and its limit is independent of the forms $\ell$ and $\varepsilon$.
\end{coro}

In summary, we consider Theorems \ref{Thm1}, \ref{keyco}, \ref{Thm2} and their corollaries to be the heart in our conceptual setting of equivariant motivic integration, which inherits Denef-Loeser's idea on the classical motivic integration for stable semi-algebraic subsets of an arc space. Unless taking the rationality result into account, we have developed the equivariant motivic integration with {\it any} group scheme being the limit of a projective system of finite algebraic groups.
For formal schemes, Hartmann \cite{Hartmann} recently provided a preprint on equivariant motivic integration with respect to an abelian finite group. Her work may be regarded as the version with action of finite groups of Sebag and Loeser's motivic integration \cite{Se}, \cite{LS}.

\subsection{}
We discuss a crucial application of the $\hat\mu$-equivariant motivic integration to the full version of the integral identity conjecture for regular functions. It is well known that this conjecture is a building block in Kontsevich-Soibelman's theory of motivic Donaldson-Thomas invariants for noncommutative Calabi-Yau threefolds, it implies directly the existence of these invariants (cf. \cite{KS}, \cite[Section 1]{Thuong2}). Let us state the version for regular functions of the conjecture (for the version for formal functions, see \cite[Conjecture 4.4]{KS}).

Recall from \cite{DL1} that any morphism of $k$-varieties $g: S\to S'$ induces a morphism of rings $g^*: \mathscr M_{S'}^{\hat\mu}\to \mathscr M_S^{\hat\mu}$ by fiber product, and induces a morphism of groups $g_!: \mathscr M_S^{\hat\mu}\to \mathscr M_{S'}^{\hat\mu}$ by composition. When $S'$ is $\Spec k$ we write $\int_S$ for $g_!$.

\begin{conjj}[Kontsevich-Soibelman]\label{conj}
Let $(x,y,z)$ be the standard coordinates of the vector space $k^d=k^{d_1}\times k^{d_2}\times k^{d_3}$. Let $f$ be in $k[x,y,z]$ such that $f(0,0,0)=0$ and $f(\lambda x, \lambda^{-1} y,z)=f(x,y,z)$ for all $\lambda$ in $\mathbb G_{m,k}$. Then the integral identity $\int_{\mathbb A_k^{d_1}}i^*\mathscr{S}_f=\L^{d_1}\mathscr{S}_{\tilde{f},0}$ holds in $\mathscr M_k^{\hat\mu}$, where $\tilde{f}$ is the restriction of $f$ to $\mathbb A_k^{d_3}$, and $i$ is the inclusion of $\mathbb A_k^{d_1}$ in $f^{-1}(0)$.
\end{conjj}

The conjecture was first proved by L\^e \cite{Thuong1} for the case where $f$ is either a function of Steenbrink type or the composition of a pair of regular functions with a polynomial in two variables. In \cite[Theorem 1.2]{Thuong}, in view of the formalism of Hrushovski-Kazhdan \cite{HK} and Hrushovski-Loeser \cite{HL}, L\^e showed that Conjecture \ref{conj} holds in $\mathscr M_{\loc}^{\hat\mu}$, a ``big'' localization of $\mathscr M_k^{\hat\mu}$, as soon as the base field $k$ is algebraically closed (see \cite{Thuong2} for a short review about that work). 
Recently, Nicaise and Payne \cite{NP} have developed an effective new method to compute motivic nearby cycles as motivic volume of semi-algebraic sets, the motivic Fubini theorem for the tropicalization map, on the foundation of \cite{HK} and tropical geometry. Their approach needs the condition that $k$ contains all roots of unity, but it brings out a strong improvement when proving the conjecture with the context $\mathscr M_k^{\hat\mu}$. 

Our proof of Conjecture \ref{conj} in this article is complete, no additional hypothesis is needed. The work is based on our equivariant motivic integration and it is devoted in Section \ref{lastsection} for performance of detailed arguments.

\section{Some results on the equivariance}\label{Prel}
\subsection{Equivariant Grothendieck rings of varieties}
Let $k$ be a field of characteristic zero, and $S$ an algebraic $k$-variety. As usual (cf. \cite{DL1}, \cite{DL2}), we denote by $\Var_S$ the category of $S$-varieties and $K_0(\Var_S)$ its Grothendieck ring. By definition, $K_0(\Var_S)$ is the quotient of the free abelian group generated by the $S$-isomorphism classes $[X\to S]$ in $\Var_S$ modulo the following relation 
$$[X\to S]=[Y\to S]+[X\setminus Y\to S]$$ 
for $Y$ being Zariski closed in $X$. Together with fiber product over $S$, $K_0(\Var_S)$ is a commutative ring with unity $1=[\Id: S\to S]$. Put $\L:=[\mathbb A_k^1\times_k S\to S]$ and write $\mathscr M_S$ for the localization of $K_0(\Var_S)$ which makes $\L$ invertible.

Let $X$ be an algebraic $k$-variety, and let $G$ be an algebraic group which acts on $X$. The $G$-action is called {\it good} if every $G$-orbit is contained in an affine open subset of $X$. Now we fix a good action of $G$ on the $k$-variety $S$ (we may choose the trivial action). By definition, the $G$-equivariant Grothendieck group $K_0^G(\Var_S)$ of $G$-equivariant morphism of $k$-varieties $X\to S$, where $X$ is endowed with a good $G$-action, is the quotient of the free abelian group generated by the $G$-equivariant isomorphism classes $[X\to S,\sigma]$ modulo the following relations
$$[X\to S,\sigma]=[Y\to S,\sigma|_Y]+[X\setminus Y\to S,\sigma|_{X\setminus Y}]$$
for $Y$ being $\sigma$-invariant Zariski closed in $X$, and moreover
\begin{align}\label{equiv}
[X\times_k\mathbb A_k^n\to S,\sigma]=[X\times_k\mathbb A_k^n\to S,\sigma']
\end{align}
if $\sigma$ and $\sigma'$ lift the same good $G$-action on $X$. As above, we have the commutative ring with unity structure on $K_0^G(\Var_S)$ by fiber product, where $G$-action on the fiber product is through the diagonal $G$-action, and we may define the localization $\mathscr M_S^G$ of the ring $K_0^G(\Var_S)$ by inverting $\L$. Here, $\L$ is the class of $[\mathbb A_k^1\times_k S\to S]$ endowed with a good action of $G$. 

Let $\hat G$ be a group $k$-scheme of the form $\hat G=\varprojlim_{i\in I} G_i$, where $I$ is a partially ordered set and $\left\{G_i, G_j\to G_i \mid i\leq j \ \text{in}\ I\right\}$ is a projective system of algebraic groups over $k$. In particular, we may consider $\hat G$ to be the profinite group scheme of roots of unity $\hat\mu$, the projective limit of the group schemes $\mu_n=\Spec \left(k[T]/(T^n-1)\right)$ and transition morphisms $\mu_{mn}\to \mu_n$ sending $\lambda$ to $\lambda^m$. We define $K_0^{\hat G}(\Var_S):=\varinjlim_{i\in I}  K_0^{G_i}(\Var_S)$ and $\mathscr M_S^{\hat G}:=K_0^{\hat G}(\Var_S)[\L^{-1}]$, which implies the identity $\mathscr M_S^{\hat G}=\varinjlim_{i\in I}  \mathscr M_S^{G_i}$.  

\subsection{Equivariant piecewise trivial fibrations}
We start with Denef-Loeser's definition of piecewise trivial fibration in \cite{DL2}. Let $X$, $Y$ and $F$ be algebraic $k$-varieties, and let $A$ and $B$ be respectively constructible subsets of $X$ and $Y$. A map $f: A\to B$ is called a {\it piecewise trivial fibration with fiber $F$} if there is a finite partition of $B$ into locally closed subsets $B_i$ in $Y$ such that, for every $i$, $f^{-1}(B_i)$ is locally closed in $X$ and isomorphic as a $k$-variety to $B_i\times_k F$, and that under the isomorphism $f|_{B_i}$ corresponds to the projection $B_i\times_k F\to B_i$. More generally, given a constructible subset $C$ of $B$, $f$ is by definition a {\it piecewise trivial fibration over $C$} if $f|_{f^{-1}(C)}: f^{-1}(C) \to C$ is a piecewise trivial fibration.  

\begin{theorem}[Sebag \cite{Se}, Th\'eor\`eme 4.2.3]
With the previous notation and hypotheses, the map $f: A\to B$ is a piecewise trivial fibration with fiber $F$ if and only if for every $y$ in $B$, the fiber $f^{-1}(y)$ is isomorphic as a $k(y)$-variety to $F\times_kk(y)$.
\end{theorem}

Now we go to the equivariant framework. Let $G$ be an algebraic group over $k$. Let $X$, $Y$ and $F$ be algebraic $k$-varieties endowed with a good $G$-action, let $A$ and $B$ be $G$-invariant constructible subsets of $X$ and $Y$, respectively. Consider a $G$-equivariant map of constructible sets $f: A\to B$ which is the restriction of a given $G$-equivariant morphism $X\to Y$. Then $f$ is called a {\it $G$-equivariant piecewise trivial fibration with fiber $F$} if there exists a stratification of $B$ into finitely many $G$-invariant locally closed subsets $B_i$ such that, for every $i$, $f^{-1}(B_i)$ is a $G$-invariant constructible subset of $A$ and $f^{-1}(B_i)$ is $G$-equivariant isomorphic to $B_i\times_k F$, with the action of $G$ on $B_i\times_k F$ being the diagonal one, and moreover, over $B_i$, $f$ corresponds to the projection $B_i\times_k F\to B_i$. In view of the definition of $K_0^G(\Var_k)$, such a map $f$ yields the identity 
$$[X]=[Y]\cdot [F]$$ 
in $K_0^G(\Var_k)$. In the following theorem we will give a criterion for which a morphism of algebraic $k$-varieties with $G$-action is a $G$-equivariant piecewise trivial fibration (we only consider the case $A=X$ and $B=Y$ in the previous definition). Let us fix the notation which concern. For a morphism of algebraic $k$-varieties $X\to Y$ and any immersion $S\to Y$, the standard notation $X_S$ will stand for their fiber product. And, for each $y$ in $Y$, the {\em stabilizer subgroup $G_y$} of $G$ with respect to $y$ is the subgroup of elements in $G$ that fix $y$.
 
\begin{theorem}\label{fund2} 
Let $X$, $Y$ and $F$ be algebraic $k$-varieties endowed with good $G$-actions, and let $f\colon X\to Y$ be a $G$-equivariant morphism. Assume that the categorical quotient $Y\to Y/\!/G$ exists and it is quasi-finite (e.g., $G$ is a finite group). Then $f$ is a $G$-equivariant piecewise trivial fibration if and only if for every $y$ in $Y$, there exists an $G_y$-equivariant isomorphism of $k(y)$-varieties $X_y \stackrel{\cong}{\longrightarrow} F\times_k k(y)$. 

\end{theorem}

\begin{proof}
We first observe that, if $G$ is a finite group, then the geometric quotient $\phi_Y\colon Y\to Y/\!/G$ exists and it is a finite morphism, according to \cite[Expos\'e V, Proposition 1.8]{Gro63}. We can see moreover that the ``only if'' statement of the theorem is obvious, so we will only prove the ``if'' statement. Let $S_0$ be the set of all the generic points of $Y/\!/G$ and let $\zeta\in S_0$. Since the quotient $\phi_Y\colon Y\to Y/\!/G$ is quasi-finite, the scheme $Y_{\zeta}$ is finite. Then $Y_{\zeta}=\bigsqcup_{\eta\in I}G\cdot\eta$, where $I$ is a finite subset of $Y$.  Furthermore, the orbits $G\cdot\eta$ are also finite, that is, $G\cdot\eta=\{ \eta,\eta_1,\ldots\eta_l\}$ for some natural number $l$. For each $1\leq i\leq  l$, we take an element $g_i$ in $G$ such that $\eta_i=g_i\cdot\eta$ and define an isomorphism of $k({\eta_i})$-schemes $\Phi_{\eta_i}\colon X_{\eta_i}\cong F\times_k k({\eta_i})$ as the composition
$$X_{\eta_i}\overset{g^{-1}_i\cdot}{\longrightarrow} X_{\eta} \cong F\times k(\eta)\overset{g_{i}\cdot }{\longrightarrow} F\times k(\eta_i),$$
i.e. $\Phi_{\eta_i}(x)=g_i\Phi_{\eta}(g^{-1}_i x)$, where the middle isomorphism $\Phi_{\eta}$ is taken from the hypothesis. Combining all the isomorphisms $\Phi_{\eta_i}$ with $\zeta$ in $S_0$ and $\eta$ in $I$ we obtain an isomorphism 
$$\Phi_{S_0}\colon X_{S_0}\cong  F\times_k Y_{S_0},$$
defined as, $\Phi_{S_0}(x)=\Phi_{\eta_i}(x)$ if $x\in X_{\eta_i}$. We have, moreover, that $\Phi_{S_0}$ is $G$-equivariant. In fact, take any $x\in X_{S_0}$ and $g\in G$. We assume that $\eta_1=f(x), \eta_2=f(gx)$ and $\eta_1=g_1\eta, \eta_2=g_2\eta$ as above. Then $$g_2\eta=\eta_2=f(gx)=g f(x)=g\eta_1=gg_1\eta,$$
and therefore $g^{-1}_2gg_1\in G_{\eta}$. This implies that
\begin{align*}
\Phi_{S_0}(gx)=\Phi_{\eta_2}(gx)=g_2\Phi_{\eta}(g^{-1}_2gx)=g_2\Phi_{\eta}(g^{-1}_2gg_1 g^{-1}_1x)=g_2g^{-1}_2gg_1\Phi_{\eta}( g^{-1}_1x),
\end{align*}
where the last equality follows from the $G_{\eta}$-equivariance of $\Phi_{\eta}$. Hence
\begin{align*}
\Phi_{S_0}(gx)=gg_1\Phi_{\eta}( g^{-1}_1x)=g\Phi_{\eta_1}(x)=g \Phi_{S_0}(x),
\end{align*}
as desired. We are going to show that there exists a dense open subset $U$ of $Y/\!/G$ such that there is an isomorphism $X_U\cong U\times F$ through which $f$ factors. Indeed, let $\mathcal{T}$ be the set of open dense subsets $\lambda$ of $Y/\!/G$. It is a directed partially ordered set with relation $\lambda\leq \lambda'$ if $\lambda'\subseteq \lambda$. To apply Grothendieck's descending theory we define the initial objects 
$$\alpha:=Y/\!/G\in \mathcal{T},\quad S_{\alpha}:=Y,$$ 
and $S_{\alpha}$-schemes 
$$A_{\alpha}:=X\overset{f}{\longrightarrow}Y,\quad B_{\alpha}:=Y\times F\overset{\pr_1}{\longrightarrow}Y,$$ 
and consider the following projective systems with natural transition morphisms
$$S_{\lambda}:=Y_{\lambda}, \quad A_{\lambda}:=A_{\alpha}\times_{S_{\alpha}}S_{\lambda}=X_{\lambda},\quad B_{\lambda}:=B_{\alpha}\times_{S_{\alpha}}S_{\lambda}=Y_{\lambda}\times F.$$
Note that $\varprojlim S_{\lambda}=Y_{S_0}$, $A:=\varprojlim A_{\lambda}=X_{S_0}$ and $B:=\varprojlim B_{\lambda}=Y_{S_0}\times F$. It follows from \cite[Corollary 8.8.2.5]{Gro66} (see also, \cite{Stack}, Lemma 31.8.11 (Tag 081E)) that there exist $U\geq \alpha$ and an $S_U$-isomorphism $\Phi_U\colon A_{U}\to B_{U}$ such that the diagram 
\begin{displaymath}
\xymatrix{
X_{U} \ar[rr]^{\Phi_{U}}\ar@{->}[dr]_{f}&&Y_{U}\times F\ar@{->}[dl]^{\pr_1}\\
&Y_{U}&
}
\end{displaymath}
commutes. Let us now consider the two $Y_{U}$-schemes $G\times X_U $ and $Y_U\times F$, and the two $Y_{U}$-morphisms of these schemes $\varphi_U=\Phi_{U}\circ \rho_1$ and $\psi_U=\rho_2\circ (\Id \times \Phi_{U})$
\begin{displaymath}
\xymatrix{
&X_{U} \ar[dr]^{\Phi_{U}}\ar@{<-}[dl]_{\rho_1}&\\
G\times X_{U} \ar@{->}[dr]_{\Id \times \Phi_{U}}&&Y_{U}\times F\ar@{<-}[dl]^{\rho_2}\\
&G\times Y_{U}\times F&
}
\end{displaymath}
where $\rho_1$ and $\rho_2$ are action morphisms of $G$ on $X_{U}$ and $Y_{U}\times F$, respectively. It follows from the $G$-equivariant isomorphism $X_{S_0}\cong Y_{S_0}\times F$ that $\varphi_{U,Y_{S_0}}=\psi_{U,Y_{S_0}}$, where $\varphi_{U,Y_{S_0}}$ and $\psi_{U,Y_{S_0}}$ are the base changes of $\varphi_{U}$ and $\psi_{U}$ by the inclusion $Y_{S_0}\to Y_{U}$, respectively. Applying \cite{Stack}, Lemma 31.10.1 (Tag 01ZM), there exists $V \geq U$ such that $\varphi_{U,Y_{V}}=\psi_{U,Y_{V}}$. Note that $Y_V$ is a $G$-invariant subset of $Y_U$, it then follows that $\varphi_{U,Y_{V}}=\varphi_V$ and $\psi_{U,Y_{V}}=\psi_V$, where $\varphi_V=\Phi_{V}\circ \rho_1$ and $\psi_V=\rho_2\circ (\Id \times \Phi_{V})$. The equality $\varphi_V=\psi_V$ yields that the isomorphism 
$$\Phi_V\colon X_{V}\to Y_{V}\times F$$ 
is $G$-equivariant. Repeating the above argument for the closed subset $Y\setminus Y_V$ and so on, we get a finite stratification $Y=\bigsqcup_{0\leq i\leq n_0} Y_i$ into $G$-invariant locally closed subsets $Y_i$, with $Y_1=Y_V$ and $\dim Y_i>\dim Y_{i+1}$, such that $f$ is a $G$-equivariant trivial fibration with fiber $F$ over each stratum $Y_i$. This completes the proof of the theorem. 
\end{proof}

\begin{theorem}\label{keycoro}
Let $X$ and $Y$ be algebraic $k$-varieties endowed with good $G$-action, and let $f\colon X\to Y$ be a $G$-equivariant morphism. Assume the categorical quotient $Y\to Y/\!/G$ exists and it is quasi-finite (e.g., $G$ is a finite group). Assume moreover that for every $y$ in $Y$, there exists an isomorphism of $k(y)$-varieties $X_y\cong \mathbb{A}^n_{k(y)}$, for a given $n$ in $\mathbb N$.
Then the identity 
$$ [X]=[Y]\cdot \mathbb{L}^n$$
holds in $K_0^G(\Var_k)$.
\end{theorem}

\begin{proof}
As in the proof of Theorem \ref{fund2}, it follows from \cite[Expos\'e V, Proposition 1.8]{Gro63} that, if $G$ is finite, then the geometric quotient $\phi_Y\colon Y\to Y/\!/G$ exists and it is a finite morphism. Let $S_0$ be the set of the generic points of $Y/\!/G$. Then the scheme $Y_{S_0}$ is $G$-invariant and it is a finite subset of $Y$. By assumption, for each $\eta$ in $Y_{S_0}$, there exists an isomorphism $\Phi_{\eta}\colon X_{\eta}\to k(\eta)\times \mathbb{A}^n_k$. It yields an isomorphism $\Phi_{S_0} \colon X_{S_0}\to Y_{S_0}\times \mathbb{A}^n_k$ defined by $\Phi_{S_0}(x)=\Phi_{\eta}(x)$ if $x$ is in $X_{\eta}$. Applying \cite[Corollary 8.8.2.5]{Gro66} (see the first part of the proof of Theorem \ref{fund2}) we obtain an open dense subset $V$ in $Y/\!/G$ and an isomorphism $\Phi_{V}\colon X_{V}\cong Y_{V}\times \mathbb{A}^n_{k}$ of $Y_{V}$-varieties, through which $f$ factors, i.e $f=\pi_1\circ \Phi_{V}$. We now endow the product variety $Y_{V}\times \mathbb{A}^n_k$ with a good action of $G$ defined as $g\cdot z=\Phi_{V}\left(g\cdot\Phi^{-1}_{V}(z)\right)$ for all $g\in G$ and $z\in Y_{V}\times \mathbb{A}^n_k$. Then the isomorphism $\Phi_{V}$ and the projection (on to the first component) $\pi_1$ becomes $G$-equivariant. It yields that, the action on $Y_{V}\times \mathbb{A}^n_k$ is a lifting of the action of $G$ on $Y_{V}$. We deduce the identity 
$$[X_{V}]= [Y_{V}\times\mathbb{A}^n_{k}],$$
which is equal to $[Y_{V}]\cdot \mathbb{L}^n$  
in $K_0^G(\Var_k)$, by the definition of the ring $K_0^G(\Var_k)$ (in particular, see relation \eqref{equiv}). Note that $V$ is open and dense in $Y/\!/G$, so its complement $(Y/\!/G)\setminus V$ is closed subset of $Y/\!/G$. Repeating the above argument for the closed subset $(Y/\!/G)\setminus V$ and so on, we get a stratification $Y/\!/G=\bigsqcup_{0\leq i\leq m_0} Z_i$ of $G$-invariant locally closed subsets $Z_i$, with $Z_1=V$ and $\dim Z_i>\dim Z_{i+1}$, such that the identities 
$$[X_{Z_i}]= [Y_{Z_i}]\cdot \mathbb{L}^n$$ 
hold in $K_0^G(\Var_k)$. The theorem is now definitely proved.
\end{proof}


\section{Arc spaces, equivariant motivic measure and integration}\label{arc-motint}
\subsection{Arc spaces and semi-algebraic sets}
Let $X$ be an algebraic $k$-variety. For any $n$ in $\mathbb N$, denote by $\mathscr L_n(X)$ the $k$-scheme of $n$-jets of $X$, which represents the functor from the category of $k$-algebras to the category of sets sending a $k$-algebra $A$ to $\Mor_{k\text{-schemes}}(\Spec(A[t]/A(t^{n+1})),X)$. For $m\geq n$, the truncation $k[t]/(t^{m+1})\to k[t]/(t^{n+1})$ induces a morphism of $k$-schemes 
$$\pi_n^m:\mathscr L_m(X)\to \mathscr L_n(X)$$ 
and this is an affine morphism. If $X$ is smooth of dimension $d$, the morphism $\pi_n^m$ is a locally trivial fibration with fiber $\mathbb{A}_k^{(m-n)d}$. The $n$-jet schemes and the morphisms $\pi_n^m$ form in a natural way a projective system of $k$-schemes, we call the projective limit 
$$\mathscr L(X):=\varprojlim \mathscr L_n(X)$$ 
the {\it arc space of $X$}. Note that $\mathscr L(X)$ is a $k$-scheme but it is not of finite type. For any field extension $K\supseteq k$, the $K$-points of $\mathscr{L}(X)$ correspond one-to-one to the $K[[t]]$-points of $X$. Denote by $\pi_n$ the natural morphism 
$$\mathscr{L}(X)\to\mathscr{L}_n(X).$$ 

Recall from \cite[Section 2]{DL2}, for any algebraically closed field $K$ containing $k$, that a subset of $K(\!(t)\!)^m\times \mathbb Z^r$ is {\it semi-algebraic} if it is a finite boolean combination of sets of the form
\begin{align*}
\big\{(x,\alpha)\in & K(\!(t)\!)^m\times \mathbb Z^r \mid \ord_tf_1(x)\geq \ord_tf_2(x)+\ell_1(\alpha),\\ 
&\ord_tf_3(x)\equiv \ell_2(\alpha)\mod d,\ \Phi(\overline{\ac}(g_1(x)),\dots,\overline{\ac}(g_n(x)))=0\big\},
\end{align*}
where $f_i$, $g_j$ and $\Phi$ are polynomials over $k$, $\ell_1$ and $\ell_2$ are polynomials over $\mathbb Z$ of degree at most 1, $d$ is in $\mathbb N$, and $\overline{\ac}(g_j(x))$ is the angular component of $g_j(x)$. One calls a collection of formulas defining a semi-algebraic set a {\it semi-algebraic condition}. A family $\{A_{\alpha} \mid \alpha\in \mathbb N^r\}$ of subsets $A_{\alpha}$ of $\mathscr L(X)$ is called a {\it semi-algebraic family of semi-algebraic subsets} if there exists a covering of $X$ by affine Zariski open sets $U$ such that $A_{\alpha}\cap \mathscr L(U)$ is defined by a semi-algebraic condition, that is, 
\begin{align*}
A_{\alpha}\cap \mathscr L(U)=\{\gamma\in \mathscr L(U) \mid \theta(h_1(\tilde{\gamma}),\dots,h_m(\tilde{\gamma});\alpha)\},
\end{align*}
where $h_i$ are regular functions on $U$, $\theta$ is a semi-algebraic condition, and $\tilde{\gamma}$ is the element in $\mathscr L(U)(k(\gamma))$ corresponding to a point $\gamma$ in $\mathscr L(U)$ of residue field $k(\gamma)$ (cf. \cite[Section 2.2]{DL2}). In the case when $r=0$, the unique element in the previous family is called a {\it semi-algebraic subset} of $\mathscr L(X)$.

Let $A$ be a semi-algebraic subset of $\mathscr L(X)$, and let $r$ be in $\mathbb N$. As introduced in \cite{DL2}, a function 
$$\ell: A\times \mathbb Z^r\to \mathbb Z\cup \{+\infty\}$$ 
is called {\it simple} if the family of sets $\left\{x\in A \mid \ell(x,\alpha_1,\dots,\alpha_r)=\alpha_{r+1}\right\}$, with $(\alpha_1,\dots,\alpha_{r+1})$ in $\mathbb N^{r+1}$, is a semi-algebraic family of semi-algebraic subsets of $\mathscr L(X)$. For instance, if $f$ is a regular function on $X$ and $A$ is a semi-algebraic subset of $\mathscr L(X)$, then $\ord_tf$ is a simple function on $A$. A subset of $\mathbb Z^r$ is called a {\it Presburger set} if it is defined by a finite boolean combination of sets of the form
$$\big\{\alpha\in \mathbb Z^r \mid \ell_1(\alpha)\geq 0,\ \ell_2(\alpha)\equiv 0\mod d\big\},$$
where $\ell_1$ and $\ell_2$ are polynomials over $\mathbb Z$ of degree at most 1 and $d$ is in $\mathbb N_{>0}$. In other words, a Presburger set is a subset of $\mathbb Z^r$ (for some $r$) defined by a formula in the Presburger language. If $\ell$ is a  $\mathbb Z$-valued function on $\mathbb Z^r$ whose graph is a Presburger subset of $\mathbb Z^{r+1}$, then we call $\ell$ a {\it Presburger function}.

\subsection{Equivariant motivic measure and integration}\label{EquivMotInt}
Let $X$ be an algebraic $k$-variety of pure dimenstion $d$, and $A$ a semi-algebraic subset of the arc space $\mathscr L(X)$. According to \cite{DL2}, $A$ is called {\it weakly stable at level $n$}, for $n$ in $\mathbb N$, if $A$ is a union of fibers of $\pi_n: \mathscr L(X) \to \mathscr L_n(X)$, and $A$ is called {\it weakly stable} if it is weakly stable at some level. Further, $A$ is called {\it stable at level $n$} if it is weakly stable at level $n$ and, for every $m\geq n$, the map 
\begin{align}\label{pimap}
\pi_{m+1}(\mathscr L(X))\to \pi_m(\mathscr L X))
\end{align}
is a piecewise trivial fibration over $\pi_m(A)$ with fiber $\mathbb A_k^d$; and $A$ is {\it stable} if it is stable at some level. Let $\mathscr F_X$ be the family of stable semi-algebraic subsets of $\mathscr L(X)$. Note that if $A$ is a weakly stable semi-algebraic subset of $\mathscr L(X)$ and $A$ is disjoint with $\mathscr L(X_{\Sing})$, then $A$ is in $\mathscr F_X$, where $X_{\Sing}$ is the locus of singular points of $X$. As noticed in \cite{DL2}, the family $\mathscr F_X$ is closed for finite intersection and finite union operations. A direct corollary of the definition is that if $A$ is in $\mathscr F_X$, then there exists a natural number $n$ such that, for every $m\geq n$, the identity $[\pi_m(A)]=[\pi_n(A)]\L^{(m-n)d}$ holds in $K_0(\Var_k)$, i.e., the identity 
$$[\pi_m(A)]\L^{-md}=[\pi_n(A)]\L^{-nd}$$ 
holds in $\mathscr M_k$. One puts  
$$\tilde\mu(A):=[\pi_{n}(A)]\L^{-(n+1)d},$$ 
for $A$ in $\mathscr F_X$ and $n$ the stable level of $A$, and obtains an additive measure $\tilde\mu: \mathscr F_X \to \mathscr M_k$.

Let $\hat{G}=\varprojlim G_i$ be the limit of a projective system of finite algebraic groups over a directed ordered set $(I,\leq)$. Assume $\hat{G}$ acts nicely on $\mathscr L(X)$, that is, the given action of $G_i$ on $\mathscr L_n(X)$ are good for every $i$ in $I$ and $m\geq n$, and the morphisms 
$$\pi^m_n: \mathscr L_m(X) \to \mathscr L_n(X)$$ 
are $G_i$-equivariant. Let $i$ be in $I$, and let $A$ be a semi-algebraic subset of $\mathscr L(X)$ which is {\em $G_i$-invariant stable at level $n$}, i.e., $A$ is stable at level $n$ and $\pi_m(A)$ is invariant under the action of $G_i$ for all $m\geq n$. Then the morphism (\ref{pimap}) is $G_i$-equivariant and it is a piecewise trivial fibration with fiber $\mathbb{A}^d_k$ for all $m\geq n$. By Theorem \ref{keycoro}, the identities
$$[\pi_{m+1}(A)]=[\pi_m(A)]\L^{d}$$ holds in $K_0^{G_i}(\Var_k)$, and hence $[\pi_m(A)]\L^{-md}$ is constant in $\mathscr M_k^{G_i}$ for every $m\geq n$. Putting $\tilde\mu^{G_i}(A):=[\pi_n(A)]\L^{-(n+1)d}$ we also get a $G_i$-equivariant additive measure $\tilde\mu^{G_i}: \mathscr F_X^{G_i} \to \mathscr M_k^{G_i},$ where $\mathscr F_X^{G_i}$ denotes the subfamily of $\mathscr F_X$ consisting of ${G_i}$-equivariant stable semi-algebraic subsets of $\mathscr L(X)$. The family $\{\tilde\mu^{G_i},i\in I\}$ then forms a projective system and its projective limit defines a {\em $\hat{G}$-equivariant additive measure}
\begin{align}\label{measure!}
\tilde\mu^{\hat{G}}: \mathscr F_X^{\hat{G}} \to \mathscr M_k^{\hat{G}},
\end{align}
where $\mathscr F_X^{\hat{G}}$ is a projective limit of the system $\mathscr F_X^{G_i}$. When $\hat G$ is $\hat{\mu}$, with $I=\mathbb N$, and it acts naturally on $\mathscr L(X)$ as $\lambda\cdot \gamma(t):= \gamma(\lambda t)$ for every $\lambda$ in $\mu_n$ and $\gamma$ in $\mathscr L_n(X)$, we shall write simply $\tilde\mu$ instead of $\tilde\mu^{\hat{\mu}}$.

Now, let $A$ be in $\mathscr F_X^{\hat{G}}$ and let $\ell: A\to \mathbb N$ be a simple function such that all the fibers $\ell^{-1}(n)$ of $\ell$ are in $\mathscr F_X^{\hat{G}}$. By \cite[Lemma 2.4]{DL2}, $A$ is the disjoint union of finitely many subsets $\ell^{-1}(n)$. Then we may define {\it ${\hat{G}}$-equivariant motivic integral} of $\ell$ to be
\begin{align}\label{integral}
\int_A\L^{-\ell}d\tilde\mu^{\hat{G}}:=\sum_{n\in \mathbb N}\tilde\mu^{\hat{G}}\left(\ell^{-1}(n)\right)\L^{-n},
\end{align}
which takes value in $\mathscr M_k^{\hat{G}}$.

As in \cite[Section 3.3]{DL2}, we can define the order $\ord_t\mathcal J$ of a coherent sheaf of ideals $\mathcal J$ on an algebraic $k$-variety $Z$ of pure dimension $d$, which is a simple function. Denote by $\Omega_Z^1$ the sheaf of differentials on $Z$, and by $\Omega_Z^d$ the $d$th exterior power of $\Omega_Z^1$. Let $\mathcal S$ be a coherent sheaf on $Z$ such that there exists a natural morphism of sheaves $\iota: \mathcal S \to \Omega_Z^d$. Assume that $Z$ is smooth. Let $\mathcal J(\mathcal S)$ be the sheaf of ideals on $Z$ locally generated by functions $\iota(s)/dz$ with $s$ a local section of $\mathcal S$ and $dz$ a local volume form on $Z$. Then we define $\ord_t\mathcal S:=\ord_t\mathcal J(\mathcal S)$.

\begin{theorem}\label{changeofvariable}
Let $X$ and $Y$ be algebraic $k$-varieties of pure dimension $d$, with $Y$ being smooth in addition. Let $h:Y\to X$ be a proper birational morphism. Let $\hat{G}=\varprojlim G_i$ act nicely on $\mathscr L(Y)$ and on $\mathscr L(X)$ such that all the morphisms $h_n: \mathscr L_n(Y) \to \mathscr L_n(X)$ are $G_i$-equivariant. Let $A$ be in $\mathscr F_X^{\hat{G}}$ such that $A\cap \mathscr L\left(h(E)\right)=\emptyset$, where $E$ is the exceptional locus of $h$. Let $\ell: A\to \mathbb N$ a simple function whose fibers are all in $\mathscr F_X^{\hat{G}}$. Then $h^{-1}(A)$ and the fibers of $\ell\circ h+\ord_th^*(\Omega_X^d)$ on $h^{-1}(A)$ are in $\mathscr F_Y^{\hat{G}}$, and furthermore, the identity
$$\int_A\L^{-\ell}d\tilde\mu^{\hat{G}}=\int_{h^{-1}(A)}\L^{-\ell\circ h-\ord_th^*(\Omega_X^d)}d\tilde\mu^{\hat{G}}$$
holds in $\mathscr M_k^{\hat{G}}$.
\end{theorem}

\begin{proof}
We assume that $A\in \mathscr F_X^{G_i}$ for some $i\in I$. The first statment that $h^{-1}(A)$ and the fibers of $\ell\circ h+\ord_th^*(\Omega_X^d)$ on $h^{-1}(A)$ are in $\mathscr F_Y^{G_i}$ is clear by \cite[Lemma 3.3]{DL2} any by the hypothesis that all $h_n$ are $G_i$-equivariant. Now, for simplicity of notation, we shall write $\tilde{A}$ for $h^{-1}(A)$, and write $\tilde{\ell}$ (resp. $\nu$) for the simple function $\ell\circ h+\ord_th^*(\Omega_X^d)$ (resp. $\ord_th^*(\Omega_X^d)$). Since $\tilde{A}$ and all the fibers of $\tilde{\ell}$ are stable, it follows from \cite[Lemma 2.4]{DL2} that the functions $\tilde{\ell}$ and $\nu$ are bounded. Choose a positive integer $N$ such that all the fibers of $\ell$ are stable at level $N$ and that $|\tilde{\ell}|\leq \frac N 2$. Define 
$$
A_n:=\pi_N(\ell^{-1}(n)), \ \tilde{A}_n:=h_N^{-1}(A_n),\ \tilde{A}_{n,e}:=\tilde{A}_n\cap \nu^{-1}(e),\ {A}_{n,e}:=h(\tilde{A}_{n,e}),
$$
for every $e\geq 0$. Then we have that $\tilde{A}_n=\pi_N((\ell\circ h)^{-1}(n))$ and that, by \cite[Lemma 3.4]{DL2}, the morphism $h|_{\tilde{A}_{n,e}}: \tilde{A}_{n,e}\to {A}_{n,e}$ is a piecewise trivial fibration with fiber $\mathbb{A}^e_k$. We then deduce from Theorem \ref{keycoro} that $[\tilde{A}_{n,e}]= [{A}_{n,e}]\L^e$ in $\mathscr M_k^{G_i}$. Therefore, we get 
\begin{align*}
\int_A\L^{-\ell}d\tilde\mu^{G_i} &=\sum_{n}[A_n]\L^{-(N+1)d-n}=\sum_{n,e}[A_{n,e}]\L^{-(N+1)d-n}\\
&\qquad =\sum_{n,e}[\tilde{A}_{n,e}]\L^{-(N+1)d-(n+e)}=\sum_{m}\left(\sum_{n+e=m}[\tilde{A}_{n,e}]\right)\L^{-(N+1)d-m}\\
&\qquad\qquad =\sum_{m}\left[\pi_N\left(\tilde{\ell}^{-1}(m)\right)\right]\L^{-(N+1)d-m}=\int_{h^{-1}(A)}\L^{-\tilde{\ell}} d\tilde\mu^{G_i},
\end{align*}
as desired.
\end{proof}

\section{Rationality of generalized motivic zeta functions}\label{rationality}
\subsection{Formal series and Hadamard product}
Let $\mathscr M$ be a commutative ring with unity which contains $\L$ and $\L^{-1}$, and let $\mathscr M[[T]]$ be the set of formal power series in $T$ with coefficients in $\mathscr M$, which is a ring and also a $\mathscr M$-module with respect to usual operations for series. Denote by $\mathscr M[[T]]_{\sr}$ the submodule of $\mathscr M[[T]]$ generated by 1 and by finite products of terms $\frac{\L^aT^b}{(1-\L^aT^b)}$ for $(a,b)$ in $\mathbb{Z}\times\mathbb{N}_{>0}$. An element of $\mathscr M[[T]]_{\sr}$ is called a {\it rational} series. By \cite{DL1}, there exists a unique $\mathscr M$-linear morphism 
$$\lim_{T\to\infty}: \mathscr M[[T]]_{\sr}\to \mathscr M$$ 
such that 
$$\lim_{T\to\infty}\frac{\L^aT^b}{(1-\L^aT^b)}=-1$$ 
for any $(a,b)$ in $\mathbb{Z}\times\mathbb{N}_{>0}$. 

More generally, we also consider the ring of formal power series $\mathscr M[[T_1,\dots,T_r]]$ in $r$ variables $(T_1,\dots,T_r)$, and its subset $\mathscr M[[T_1,\dots,T_r]]_{\sr}$ the polynomial ring with coefficients in $\mathscr M$ and in variables $\frac{\L^aT_1^{b_1}\cdots T_r^{b_r}}{1-\L^aT_1^{b_1}\cdots T_r^{b_r}}$ for $(a,b_1,\dots,b_r)$ in $\mathbb Z\times (\mathbb N^r\setminus\{(0,\dots,0)\})$. The set $\mathscr M[[T_1,\dots,T_r]]_{\sr}$ is in fact a submodule of $\mathscr M[[T_1,\dots,T_r]]$, each element of $\mathscr M[[T_1,\dots,T_r]]_{\sr}$ is called a {\it rational} series.

By definition, the Hadamard product of two formal power series $p(T)=\sum_{n\geq 1}p_nT^n$ and $q(T)=\sum_{n\geq 1}q_nT^n$ in $\mathscr M[[T]]$ is the series
\begin{align}\label{eq3.1}
p(T)\ast q(T):=\sum_{n \geq 1}p_n\cdot q_nT^n
\end{align}
in $\mathscr M[[T]]$. This product is commutative, associative, with unity $\sum_{n\geq 1}T^n$. It also preserves the rationality as seen in the following lemma.

\begin{lemma}[Looijenga \cite{Loo}, Lemma 7.6]\label{Lem2}
If $p(T)$ and $q(T)$ are rational series in $\mathscr M[[T]]$, so is $p(T)\ast q(T)$, and in this case,
$$\lim_{T\to\infty}p(T)\ast q(T)=-\lim_{T\to\infty}p(T) \cdot \lim_{T\to\infty}q(T).$$
\end{lemma}

The Hadamard product may be also defined for two formal power series in several variables. Namely, for two formal power series $p=\sum p_{n_1,\dots,n_r}T_1^{n_1}\cdots T_r^{n_r}$ and $q=\sum q_{n_1,\dots,n_r}T_1^{n_1}\cdots T_r^{n_r}$ in $\mathscr M[[T_1,\dots,T_r]]$ (the sums run over $\mathbb N^r$), we define
\begin{align}\label{eq3.2}
p\ast q:=\sum p_{n_1,\dots,n_r}\cdot q_{n_1,\dots,n_r} T_1^{n_1}\cdots T_r^{n_r},
\end{align}
which is an element of $\mathscr M[[T_1,\dots,T_r]]$. Similarly as above, the Hadamard product for formal power series in several variables is also rationality preserving, commutative, associative, and its unity is $\sum_{(n_1,\dots,n_r)\in \mathbb N^r}T_1^{n_1}\cdots T_r^{n_r}$.

\subsection{Motivic zeta functions}\label{MZF}
Let $X$ be a smooth algebraic $k$-variety of pure dimension $d$. Let $f:X\to\mathbb A_k^1$ be a regular function with the zero locus $X_0$ nonempty. For $n\geq 1$, we define 
\begin{align}\label{contactloci}
\mathscr{X}_n(f):=\left\{\gamma\in \mathscr{L}_n(X) \mid f(\gamma)= t^n\mod t^{n+1}\right\}.
\end{align}
Then $\mathscr{X}_n(f)$ is naturally an $X_0$-variety and invariant under the natural action $\sigma$ of $\mu_n$ on $\mathscr L_n(X)$ given by $\lambda\cdot\gamma(t):=\gamma(\lambda t)$. For simplicity, we write $[\mathscr{X}_n(f)]$ for the class $[\mathscr{X}_n(f)\to X_0,\sigma]$ in the ring $\mathscr M_{X_0}^{\hat{\mu}}$. The {\it motivic zeta function} of $f$ is defined to be 
\begin{align}\label{eq3.6}
Z_f(T):=\sum_{n\geq 1}[\mathscr{X}_n(f)]\L^{-nd}T^n, 
\end{align}
which is a formal power series in $\mathscr M_{X_0}^{\hat{\mu}}[[T]]$. 
If $x$ is a closed point in $X_0$, by setting 
$$\mathscr{X}_{n,x}(f)=\{\gamma\in \mathscr{X}_n(f)\mid \gamma(0)=x\}$$ 
we obtain in the same way the {\it motivic zeta function of $f$ at $x$}  
\begin{align}\label{eq3.7}
Z_{f,x}(T):=\sum_{n\geq 1}[\mathscr{X}_{n,x}(f)]\L^{-nd}T^n,
\end{align} 
which is a formal power series in $\mathscr M_k^{\hat{\mu}}[[T]]$. 

\begin{remark}
We can use the new terminology and notation in Section \ref{EquivMotInt} as follows. We note that, for $n\geq 1$ and $x$ as previous, the sets $
\overline{\mathscr X}_n(f):=\left\{\gamma\in \mathscr{L}(X) \mid f(\gamma)= t^n\mod t^{n+1}\right\}$ and $\overline{\mathscr X}_{n,x}(f):=\left\{\gamma\in \overline{\mathscr X}_n(f), \gamma(0)=x\right\}$ are in the family $\mathscr F_X^{\mu_n}$, they are stable at level $n$; and furthermore, $\mathscr X_n(f)=\pi_n(\overline{\mathscr X}_n(f))$, $\mathscr X_{n,x}(f)=\pi_n(\overline{\mathscr X}_{n,x}(f))$, and with $\tilde\mu$ in Section \ref{EquivMotInt},
$$Z_{f,x}(T)=\L^d\sum_{n\geq 1}\tilde\mu(\overline{\mathscr X}_{n,x}(f))T^n.$$
\end{remark}

As in Denef-Loeser \cite{DL2}, \cite{DL3}, to see the rationality of the series (\ref{eq3.6}) and (\ref{eq3.7}) we consider a log-resolution $h: Y\to X$ of $X_0$. The exceptional divisors and irreducible components of the strict transform for $h$ will be denoted by $E_i$, where $i$ is in a finite set $J$. For every nonempty $I\subseteq J$, we put $E_I^{\circ}=\big(\bigcap_{i\in I}E_i\big) \setminus\bigcup_{j\not\in I}E_j,$ and consider an affine covering $\{U\}$ of $Y$ such that on each piece $U\cap E_I^{\circ}\not=\emptyset$ the pullback of $f$ has the form $u\prod_{i\in I}y_i^{N_i}$, with $u$ a unit and $y_i$ a local coordinate defining $E_i$. Denote by $m_I$ the greatest common divisor of $N_i$, with $i$ in $I$. Denef and Loeser \cite{DL1} study the unramified Galois covering $\pi_I:\widetilde{E}_I^{\circ}\to E_I^{\circ}$ with Galois group $\mu_{m_I}$ defined locally with respect to $\{U\}$ as follows 
$$\left\{(z,y)\in \mathbb{A}_k^1\times(U\cap E_I^{\circ}) \mid z^{m_I}=u(y)^{-1}\right\}.$$ 
The local pieces are glued over $\{U\}$ as in the proof of \cite[Lemma 3.2.2]{DL1} to get $\widetilde{E}_I^{\circ}$ and $\pi_I$ as mentioned, and the definition of the covering $\pi_I$ is independent of the choice of $\{U\}$. Moreover, $\widetilde{E}_I^{\circ}$ is endowed with a $\mu_{m_I}$-action by multiplication of the $z$-coordinate with elements of $\mu_{m_I}$, which gives rise to an element $[\widetilde{E}_I^{\circ}]=[\widetilde{E}_I^{\circ}\to E_I^{\circ}\to X_0]$ in $\mathscr M_{X_0}^{\hat\mu}$ (cf. \cite{DL3}). For every $i$ in $J$, we denote by $\nu_i-1$ the multiplicity of $E_i$ in the canonical divisor of $h$. 

\begin{theorem}[Denef-Loeser \cite{DL1}]
With the previous notation and hypothesis, we have 
\begin{align*}
Z_f(T)=\sum_{\emptyset\not=I\subseteq J}(\mathbb{L}-1)^{|I|-1}[\widetilde{E}_I^{\circ}]\prod_{i\in I}\frac{\L^{-\nu_i}T^{N_i}}{1-\L^{-\nu_i}T^{N_i}}.
\end{align*}
In other words, the motivic zeta function of $f$ is a rational series.
\end{theorem}

An analogous formula can be also obtained for $Z_{f,x}(T)$ in (\ref{eq3.7}), so it is a rational series. The following element of $\mathscr M_{X_0}^{\hat{\mu}}$,
$$\mathscr S_f:=-\lim_{T\to\infty}Z_f(T)=\sum_{\emptyset\not=I\subset J}(1-\mathbb{L})^{|I|-1}[\widetilde{E}_I^{\circ}],$$ 
is called the {\it motivic nearby cycles} of $f$. The element $\mathscr S_{f,x}:=-\lim_{T\to\infty}Z_{f,x}(T)$ of $\mathscr M_k^{\hat{\mu}}$, which equals $(\{x\}\hookrightarrow X_0)^*\mathscr S_f$, is called the {\it motivic Milnor fiber} of $f$ at $x$.

\subsection{Generalizations}
For simplicity of performance, we only consider generalizations of the motivic zeta functions in the case where the base variety is $\Spec k$. 
Recall that a semi-algebraic condition $\theta$ is a finite boolean combination of the conditions of the forms
\begin{equation}\label{eqadd} 
\begin{gathered}
\ord_tf_1(x)\geq \ord_tf_2(x)+\ell_1(\alpha),\quad \ord_tf_3(x)\equiv \ell_2(\alpha)\mod d,\\ 
\Phi(\overline{\ac}(g_1(x)),\dots,\overline{\ac}(g_n(x)))=0,
\end{gathered}
\end{equation}
where $f_i$, $g_j$, $\Phi$ are polynomials over $k$, $\ell_1$ and $\ell_2$ are polynomials over $\mathbb Z$ of degree $\leq 1$, $x=(x_1,\dots,x_m)$ are free variables over $K(\!(t)\!)$ (with $K$ being any algebraically closed field containing $k$), $\alpha=(\alpha_1,\dots,\alpha_r)$ are free variables over $\mathbb Z$, and $d$ is in $\mathbb N_{>0}$. Suggested from \cite[Section 14.5]{CL}, we want to consider a so-called $k[t]$-semi-algebraic condition. A {\it $k[t]$-semi-algebraic} condition $\theta'$ is defined in the same way as the above $\theta$ but with $f_i$ and $g_j$ polynomials over $k[t]$ (instead of over $k$). Note that sometimes a semi-algebraic condition may be equivalent to a $k[t]$-semi-algebraic condition, that is, they may define the same semi-algebraic subset of $\mathscr L(X)$. For instance, if $f$ is a polynomial over $k$, the semi-algebraic condition 
$$\ord_tf(x)=n \wedge \overline{\ac}f(x)=1$$ 
and the $k[t]$-semi-algebraic condition 
\begin{align}\label{eq3.8}
\ord_t(f(x)-t^n)\geq \ord_t(t^n)+1
\end{align} 
are equivalent. Let us contemporarily assume that $X=\mathbb A_k^d$. Let $A$ be a semi-algebraic subset of $\mathscr L(X)$ which is defined by a $k[t]$-semi-algebraic condition $\varphi$. For $n\geq 1$, let $\varphi[n]$ denote the $k[t]$-semi-algebraic condition obtained from $\varphi$ by replacing everywhere $t$ by $t^n$. For instance, with a polynomial $f$ over $k$, if $\varphi$ is the $k[t]$-semi-algebraic condition
$$\ord_t(f(x)-t)\geq \ord_t(t)+1,$$ 
then $\varphi[n]$ is the condition (\ref{eq3.8}). If $A$ is a stable semi-algebraic subset of $\mathscr L(X)$ which is defined by a semi-algebraic condition $\theta$, and if $\theta$ is equivalent to a $k[t]$-semi-algebraic condition $\varphi$, then the subset $A[n]$ defined by $\varphi[n]$ is also a stable semi-algebraic subset of $\mathscr L(X)$. We can extend the definition to any algebraic $k$-variety $X$ by using a covering by affine open subsets. In this case, the group $\mu_n$ acts naturally on $A[n]$ in such a way that $\lambda\cdot \gamma(t)=\gamma(\lambda t)$, so we can take the $\hat\mu$-equivariant motivic measure $\tilde{\mu}(A[n])$ of $A[n]$, which is an element of $\mathscr M_k^{\hat\mu}$.

\begin{conjecture}\label{conj2}
Let $X$ be an algebraic $k$-variety, and let $A=\{A_{\alpha}\mid \alpha=(\alpha_1,\dots,\alpha_r)\in \mathbb N^r\}$ be a semi-algebraic family of semi-algebraic subsets of $\mathscr L(X)$ where there exists a covering of $X$ by affine open subsets $U$ such that a semi-algebraic condition defining each $A_{\alpha}\cap \mathscr L(U)$ is equivalent to a $k[t]$-semi-algebraic condition. Assume $A_{\alpha}$ is weakly stable (hence stable) and disjoint with $\mathscr L(X_{\Sing})$, for every $\alpha$ in $\mathbb N^r$ (hence $A_{\alpha}[n]$ is in $\mathscr F_X^{\mu_n}$ for every $\alpha$ in $\mathbb N^r$, $n\geq 1$). Then the formal power series
$$Z_A(T_0,T_1,\dots,T_r):=\sum_{(n,\alpha)\in \mathbb N^{r+1}}\tilde{\mu}\left(A_{\alpha}[n]\right)T_0^nT_1^{\alpha_1}\cdots T_r^{\alpha_r}$$
is a rational series, i.e., an element of $\mathscr M_k^{\hat{\mu}}[[T_0,T_1,\dots,T_r]]_{\sr}$.
\end{conjecture}

In what follows we are going to prove Conjecture \ref{conj2} in a special case.
 
\begin{theorem}\label{Thm3.4}
Let $X$ be an algebraic $k$-variety and $f$ a regular function on $X$. Let $A_{\alpha}$, $\alpha$ in $\mathbb N^r$, be a semi-algebraic family of semi-algebraic subsets of $\mathscr L(X)$ such that, for every covering of $X$ by affine open subsets $U$, any semi-algebraic condition defining $A_{\alpha}\cap \mathscr L(U)$ contains only conditions of two first forms in (\ref{eqadd}). Assume that, for every $\alpha$ in $\mathbb N^r$, $A_{\alpha}$ is weakly stable (hence stable) and disjoint with $\mathscr L(X_{\Sing})$. For $n\geq 1$, we put 
$$A_{n,\alpha}:=\left\{\gamma\in A_{\alpha} \mid f(\gamma)=t^n\mod t^{n+1}\right\},$$
which is in $\mathscr F_X^{\mu_n}$ for every $\alpha$ in $\mathbb N^r$, $n\geq 1$. Then the formal power series
$$Z(T_0,T_1,\dots,T_r):=\sum_{(n,\alpha)\in \mathbb N^{r+1}}\tilde{\mu}\left(A_{n,\alpha}\right)T_0^nT_1^{\alpha_1}\cdots T_r^{\alpha_r}$$
is an element of $\mathscr M_k^{\hat{\mu}}[[T_0,T_1,\dots,T_r]]_{\sr}$.
\end{theorem}

\begin{proof}
With the same reason as in the proof of \cite[Theorem 5.1$^\prime$]{DL2}, we may assume that $X$ is smooth and affine of dimension $d$. Let $\theta$ be a semi-algebraic condition which defines $A_{\alpha}$. Let $f_i$, $1\leq i\leq m$, be all the polynomials in $k[x_1,\dots,x_e]$ (for some $e$) occurring in $\theta$ (we may assume that $X$ is a closed subvariety of $\mathbb A_k^e$). By the hypothesis, $\theta$ contains only conditions of first two forms in (\ref{eqadd}), so it is of the form 
\begin{align*}
\theta=\theta'(\ord_tf_1,\dots,\ord_tf_m,\alpha),
\end{align*}
where $\theta'$ defines a Presburger subset of $\mathbb Z^{m+r}$ (i.e., $\theta'$ is a formula in the Presburger language). 
For every $\beta=(\beta_1,\dots,\beta_m)$ in $\mathbb N^m$ and $n$ in $\mathbb N_{>0}$, we put 
$$D_{\beta}:=\left\{\gamma\in \mathscr L(X)\mid \ord_tf_i(z_1(\gamma),\dots,z_e(\gamma))=\beta_i, 1\leq i\leq m\right\}$$
and 
$$D_{n,\beta}:=\left\{\gamma\in D_{\beta} \mid f(\gamma)=t^n\mod t^{n+1}\right\},$$
where $z_j$, $1\leq j\leq e$, are regular functions on $X$. We observe that $D_{n,\beta}$ is invariant under the $\mu_n$-action $\lambda\cdot \gamma(t)=\gamma(\lambda t)$. Then we get the decomposition
$$A_{n,\alpha}=\bigsqcup_{\beta\in\mathbb N^m,\theta'(\beta,\alpha)}D_{n,\beta},$$
hence the identity
\begin{align}\label{measure}
\tilde{\mu}(A_{n,\alpha})=\sum_{\beta\in\mathbb N^m,\theta'(\beta,\alpha)}\tilde{\mu}(D_{n,\beta})
\end{align}
in $\mathscr M_k^{\hat\mu}$. Since $D_{n,\beta}$ is stable of level $N:=n+\sum_{i=1}^m\beta_i$, we have $\tilde{\mu}(D_{n,\beta})=[\pi_N(D_{n,\beta})]\L^{-Nd}$.

Define 
\begin{align}\label{smallg}
g:=f\prod_{i=1}^mf_i
\end{align}
and consider a log-resolution $h: Y\to X$ of the zero locus $X_0(g)$ of $g$. We use the notation about $h$ as in Section \ref{MZF}. In particular, we consider an affine covering $\{U\}$ of $Y$ such that, on $U\cap E_I^{\circ}\not=\emptyset$, with $h(E_I^{\circ})$ contained in $X_0(g)$, we have
\begin{align}
f\circ h=u\prod_{j\in I}y_j^{N_j(f)},\ \ f_i\circ h=u_i\prod_{j\in I}y_j^{N_j(f_i)},\ 1\leq i\leq m,
\end{align}
where $u$ and $u_i$ do not vanish on $U$, and for each $j$, $y_j$ is a local coordinate defining $E_j$. We now use the idea  in the proof of \cite[Lemma 2.5]{DL3} and slightly modify it. Consider the solutions $(k_j)_{j\in I}\in \mathbb N_{\geq 1}^I$ of the system of diophantine equations
\begin{align}\label{eq3.12}
\sum_{j\in I}k_jN_j(f)=n,\ \  \sum_{j\in I}k_jN_j(f_i)=\beta_i,\ 1\leq i\leq m.
\end{align}
When emphasising the free coefficient vector $(n,\beta)^t$ in this system (\ref{eq3.12}) we write $(\ref{eq3.12})_{n,\beta}$ for it. For such a solution $(k_j)_{j\in I}$ in $\mathbb N_{\geq 1}^I$ of (\ref{eq3.12}), we put 
$$U_{(k_j)}:=\left\{\gamma\in\mathscr L_N(U) \mid \overline{\ac}f(h_{n*}(\gamma))=1, \ord_ty_j(\gamma)=k_j, j\in I\right\}.$$
Similarly as in the proof of \cite[Lemma 2.5]{DL3} we obtain the identity
$$[U_{(k_j)}]=[V_I]\L^{Nd-\sum_{j\in I}k_j},$$
which in fact lies in $\mathscr M_k^{\hat\mu}$, where 
$$V_I:=\Big\{((c_j)_{j\in I},y)\in \mathbb G_{m,k}^I\times_k (E_I^{\circ}\cap U)\mid u(y)\prod_{j\in I}c_j^{N_j(f)}=1\Big\}$$
whose class in $\mathscr M_k^{\hat\mu}$ is nothing but $(\L-1)^{|I|-1}[\widetilde{E}_I^{\circ}\cap U]$. Using Lemma 3.4 of \cite{DL2} (or more particularly, Lemma 2.2 of \cite{DL3}), the set $U_{(k_j)}$ modified with $N$ replaced by a sufficiently large natural number $l$ and augmented by the condition $\ord_t\det\Jac_h(x)=a$ is a piecewise trivial fibration with fiber $\mathbb A_k^a$ onto a subset of $\pi_l(D_{n,\beta})$. Note that, on $U\cap E_I^{\circ}\not=\emptyset$, with $h(E_I^{\circ})\subseteq X_0(g)$, we have $\ord_t\det\Jac_h=v\prod_{j\in I}y_j^{\nu_j-1}$, with $v$ a unit on $U$. As in the proof of \cite[Theorem 2.4]{DL3}, we may glue $U_{(k_j)}$ and use the properties that $\tilde{\mu}$ is additive and that $\pi_N^l$ is a locally trivial fibration as $X$ and $Y$ are smooth. We thus get the identity
$$[\pi_N(D_{n,\beta})]=\L^{Nd}\sum_{\begin{smallmatrix}\emptyset\not=I\subseteq J\\ h(E_I^{\circ})\subseteq X_0(g)\end{smallmatrix}}(\L-1)^{|I|-1}[\widetilde{E}_I^{\circ}]\sum_{\begin{smallmatrix}(k_j)_{j\in I}\in \mathbb N_{\geq 1}^I\\ (\ref{eq3.12})_{n,\beta}\end{smallmatrix}}\L^{-\sum_{j\in I}k_j\nu_j},$$  
and hence, by (\ref{measure}),
$$\tilde{\mu}(A_{n,\alpha})=\sum_{\begin{smallmatrix}\emptyset\not=I\subseteq J\\ h(E_I^{\circ})\subseteq X_0(g)\end{smallmatrix}}(\L-1)^{|I|-1}[\widetilde{E}_I^{\circ}]\sum_{\begin{smallmatrix}\beta\in\mathbb N^m\\ \theta'(\beta,\alpha)\end{smallmatrix}}\sum_{\begin{smallmatrix}(k_j)_{j\in I}\in \mathbb N_{\geq 1}^I\\ (\ref{eq3.12})_{n,\beta}\end{smallmatrix}}\L^{-\sum_{j\in I}k_j\nu_j}$$
in the ring $\mathscr M_k^{\hat\mu}$. It follows that 
\begin{align}
Z(T_0,T_1,\dots,T_r)=\sum_{\emptyset\not=I\subseteq J,\ h(E_I^{\circ})\subseteq X_0(g)}(\L-1)^{|I|-1}[\widetilde{E}_I^{\circ}] S_I(T_0,T_1,\dots,T_r),
\end{align}
where, for every $I\subseteq J$ nonempty,
\begin{align*}
S_I(T_0,T_1,\dots,T_r):=\sum_{(n,\alpha)\in \mathbb N^{r+1}}\sum_{\begin{smallmatrix}\beta\in\mathbb N^m\\ \theta'(\beta,\alpha)\end{smallmatrix}}\sum_{\begin{smallmatrix}(k_j)_{j\in I}\in \mathbb N_{>0}^I\\ (\ref{eq3.12})_{n,\beta}\end{smallmatrix}}\L^{-\sum_{j\in I}k_j\nu_j}T_0^nT_1^{\alpha_1}\cdots T_r^{\alpha_r}.
\end{align*}

Now we fix a nonempty subset $I$ of $J$ and consider the following $(r+2)$-variable formal power series 
\begin{align*}
S(T_0,T_1,\dots,T_r,T_{r+1}):=\sum_{(n,\alpha)\in \mathbb N^{r+1}}\sum_{\begin{smallmatrix}\beta\in\mathbb N^m\\ \theta'(\beta,\alpha)\end{smallmatrix}}\sum_{\begin{smallmatrix}(k_j)_{j\in I}\in \mathbb N_{>0}^I\\ (\ref{eq3.12})_{n,\beta}\end{smallmatrix}}T_0^nT_1^{\alpha_1}\cdots T_r^{\alpha_r}T_{r+1}^{\sum_{j\in I}k_j\nu_j}.
\end{align*}
Denote by $P$ the following Presburger subset of $\mathbb N^{m+r}$,
$$P:=\Big\{(n,\beta,\alpha,\alpha_{r+1})\in\mathbb N^{m+r+2}\mid \theta'(\beta,\alpha), (\ref{eq3.12})_{n,\beta},\sum_{j\in I}k_j\nu_j=\alpha_{r+1}, k_j\geq 1\ \forall j\in I\Big\}.$$ 
Taking the composition of the inclusion of $P$ in $\mathbb N^{m+r+2}$ with the projection $\mathbb N^{m+r+2}\to \mathbb N^{r+2}$ sending $(n,\beta,\alpha,\alpha_{r+1})$ to $(n,\alpha,\alpha_{r+1})$ we get a map $\rho: P\to \mathbb N^{r+2}$. Because for any $\alpha_{r+1}$ in $\mathbb N$ fixed the diophantine equation $\sum_{j\in I}k_j\nu_j=\alpha_{r+1}$ has only finitely many positive solutions, every fiber of $\rho$ is a finite set. 
By \cite[Lemma 5.2]{DL2}, the series $S(T_0,T_1,\dots,T_r,T_{r+1})$ belongs to the subring of $\mathbb Z[[T_0,T_1,\dots,T_r,T_{r+1}]]$ generated by $\mathbb Z[T_0,T_1,\dots,T_r,T_{r+1}]$ and the series $(1-T_0^{c_0}T_1^{c_1}\cdots T_r^{c_r}T_{r+1}^{c_{r+1}})^{-1}$, where $(c_0,c_1,\dots,c_r,c_{r+1})$ are in $\mathbb N^{r+2}\setminus \{(0,\dots,0)\}$. It follows that the formal power series $S_I(T_0,T_1,\dots,T_r)=S(T_0,T_1,\dots,T_r,\L^{-1})$ is a rational series, from which $Z(T_0,T_1,\dots,T_r)$ is an element of $\mathscr M_k^{\hat{\mu}}[[T_0,T_1,\dots,T_r]]_{\sr}$.
\end{proof}

\begin{proposition}\label{Prop4.5}
Let $X$, $f$, $A_{\alpha}$ and $A_{n,\alpha}$ be as in Theorem \ref{Thm3.4}. Let $\Delta$ be a rational polyhedral convex cone in $\mathbb{R}^{r+1}_{\geq 0}$ and $\bar{\Delta}$ its closure. Let $\ell$ and $\varepsilon$ be integral linear forms on $\mathbb Z^{r+1}$ with $\ell(n,\alpha)>0$ and $\varepsilon(n,\alpha)\geq 0$ 
for all $(n,\alpha)$ in $\bar{\Delta}\setminus \{0\}$. Then the formal power series
$$Z(T):=\sum_{(n,\alpha)\in \Delta\cap \mathbb N^{r+1}}\tilde{\mu}\left(A_{n,\alpha}\right)\L^{-\varepsilon(n,\alpha)}T^{\ell(n,\alpha)}$$
is an element of $\mathscr M_k^{\hat{\mu}}[[T]]_{\sr}$, and the limit $\lim_{T\to \infty}Z(T)$ is independent of such an $\ell$ and $\varepsilon$.
\end{proposition}

\begin{proof}
The first statement is direct corollary of Theorem \ref{Thm3.4}. We now prove the second one, that $\lim_{T\to \infty}Z(T)$ is independent of the linear form $\ell$. As in the proof of Theorem \ref{Thm3.4}, we may assume that $X$ is smooth of dimension $d$ and a closed subvariety of $\mathbb A_k^e$ for some $e\geq d$. Using notation and arguments in the proof of Theorem \ref{Thm3.4} we get
\begin{align*}
Z(T)=\sum_{\emptyset\not=I\subseteq J,\ h(E_I^{\circ})\subseteq X_0(g)}(\L-1)^{|I|-1}[\widetilde{E}_I^{\circ}] S_{I,\varepsilon,\ell}(T),
\end{align*}
where $g$ is as in (\ref{smallg}) and 
\begin{align*}
S_{I,\varepsilon,\ell}(T):=\sum_{(n,\alpha)\in \Delta\cap \mathbb N^{r+1}}\sum_{\begin{smallmatrix}\beta\in\mathbb N^m\\ \theta'(\beta,\alpha)\end{smallmatrix}}\sum_{\begin{smallmatrix}(k_j)_{j\in I}\in \mathbb N_{>0}^I\\ (\ref{eq3.12})_{n,\beta}\end{smallmatrix}}\L^{-\sum_{j\in I}k_j\nu_j}\L^{-\varepsilon(n,\alpha)}T^{\ell(n,\alpha)}.
\end{align*}

Let us consider the Presburger set $Q:=\left\{(\beta,\alpha)\in\mathbb N^{m+r}\mid \theta'(\beta,\alpha)\right\}$. We can assume that there is no congruence relations in the description of $Q$, because if necessary we replace $(\beta,\alpha)$ by $w(\beta,\alpha)+\delta$ for some $w$ in $\mathbb N_{>0}$ and $\delta$ in $\mathbb N^{m+r}$. Since for $Q=Q_1\cup Q_2$, the sum taking over $Q$ satisfies $\sum_Q=\sum_{Q_1}+\sum_{Q_2}-\sum_{Q_1\cap Q_2}$, we may thus assume that $Q$ is the set of integral points in a rational polyhedral convex cone $\overline{Q}$ in $\mathbb R_{\geq 0}^{m+r}$ (cf. proof of Lemma 5.2 of \cite{DL2}). We now assume that $Q$ is defined by a system of linear equations and inequations $p_u(\beta,\alpha)\geq 0$ and $q_v(\beta,\alpha)=0$ for linear forms $p_u$ and $q_v$ with integer coefficients. We add new variables $a_u$ in $\mathbb N$ and consider the system $p_u(\beta,\alpha)=a_u$, $q_v(\beta,\alpha)=0$. Then there exist rational linear forms $l_s$ such that 
$$\alpha_s=l_s(\beta',(a_u)_u,(b_w)_w),\ 1\leq s\leq r,$$ 
where $\beta'=(\beta_i)_{i\in R}$ consists of part of components of $\beta$ (i.e., $R\subseteq \{1,\dots,m\}$), which only equals $\beta$ when the maximum number of linearly independent equations in variables $\alpha$ of the system is greater than or equal to $r$, and $b_w$ are new variables in $\mathbb N$, which only appear when the maximum number of linearly independent equations in variables $\alpha$ of the system is strictly smaller than $r$. Since when replacing $\varepsilon$ and $\ell$ by $e\varepsilon$ and $e\ell$, respectively, for any $e$ in $\mathbb N_{>0}$, the dependence of $\lim_{T\to \infty}S_{I,\varepsilon,\ell}(T)$ on $\varepsilon$ and $\ell$ does not change, we may assume further that $l_s$ are integral linear forms. Now we put $\delta=((k_j)_{j\in I},(a_u)_u,(b_w)_w)$ and
$$\omega(\delta):=\ell\left(\sum_{j\in I}k_jN_j(f),\Big(l_s\Big(\Big(\sum_{j\in I}k_jN_j(f_i)\Big)_{i\in R},(a_u)_u, (b_w)_w\Big)\Big)_{1\leq s\leq r}\right)$$
and
$$
\omega'(\delta):=\sum_{j\in I}k_j\nu_j + \varepsilon\left(\sum_{j\in I}k_jN_j(f),\Big(l_s\Big(\Big(\sum_{j\in I}k_jN_j(f_i)\Big)_{i\in R},(a_u)_u, (b_w)_w\Big)\Big)_{1\leq s\leq r}\right).
$$
By the hypothesis on $\ell$, $\varepsilon$ and $l_s$, the forms $\omega(\delta)$ and $\omega'(\delta)$ are integral linear forms which are positive for all $\delta$ in $\mathbb R^N_{\geq 0}\setminus\{0\}$, where $N$ is the number of components of the vector $\delta$. On the other hand, there exists a rational polyhedral convex cone $C$ in $\mathbb R_{\geq 0}^N$ such that $(n,\alpha)$ is in $\Delta\cap \mathbb N^{r+1}$ if and only if $\delta$ is in $C\cap \mathbb N^N$. Then we have
\begin{align*}
S_{I,\varepsilon,\ell}(T)=\sum_{\delta\in C\cap \mathbb N^N}\L^{-\omega'(\delta)}T^{\omega(\delta)}.
\end{align*}
By Guibert's result \cite[Lemme 2.1.5]{G} (see also \cite[Section 2.9]{GLM1}), $\lim_{T\to\infty}S_{I,\varepsilon,\ell}(T)$ is independent of $\omega$ and $\omega'$, hence it is independent of $\ell$ and $\varepsilon$. This proves the second statement of the proposition.
\end{proof}

\section{Proof of the integral identity conjecture } \label{lastsection}

\subsection{Decomposition of the integral identity's LHS}
Let us consider the polynomial $f$ in Conjecture \ref{conj}
, which induces a regular function, also denoted by $f$, on $\mathbb A_k^d$, with the zero locus $X_0$ containing $0$ in $\mathbb A_k^d$. By the homogeneity of $f$ on the $(x,y)$-variables we have
$$f(x,0,z)=f(0,0,z)=\tilde{f}(z),$$ 
where $\tilde{f}$ is the restriction of $f$ to $\mathbb A_k^{d_3}$ (we consider $\mathbb A_k^{d_3}$ as $\{0\}\times\{0\}\times_k\mathbb A_k^{d_3}$). Then $\mathbb A_k^{d_1}$ is regarded as a $k$-subvariety of $X_0$, and the inclusion is denoted by $i$. In what follows, for simplicity of notation, we shall sometimes write $\gamma$ instead of $(x,y,z)$, for $x$, $y$ and $z$ in $\mathscr L_n(\mathbb A_k^{d_1})$, $\mathscr L_n(\mathbb A_k^{d_2})$ and $\mathscr L_n(\mathbb A_k^{d_3})$, respectively. Consider the rational series 
$$\int_{\mathbb A_k^{d_1}}i^*Z_f(T)=\sum_{n\geq 1}\int_{\mathbb A_k^{d_1}}i^*\left[\mathscr{X}_n(f)\right]\L^{-nd}T^n,$$
with coefficients in $\mathscr M_k^{\hat\mu}$. The minus limit of this rational series is nothing but the left hand side of the integral identity, namely,
\begin{align}\label{LHS}
\int_{\mathbb A_k^{d_1}}i^*\mathscr{S}_f=-\lim_{T\to\infty}\int_{\mathbb A_k^{d_1}}i^*Z_f(T).
\end{align}
Clearly, the identity
$$i^*[\mathscr{X}_n(f)]=\left[\left\{\gamma\in \mathscr{X}_n(f) \mid \gamma(0)\in\mathbb A_k^{d_1}\times_k\{0\}\times_k\{0\}\right\}\to \mathbb A_k^{d_1}, \gamma \mapsto \gamma(0)\right]$$
holds in $K_0^{\hat\mu}(\Var_{\mathbb A_k^{d_1}})$. It is convenient to use the following order of $n$-jets $x(t)$, and similarly, that of $y(t)$:
$$\ord_tx(t)=\min_{1\leq j\leq d_1}\ord_tx_j(t), \quad \ord_ty(t)=\min_{1\leq j\leq d_2}\ord_ty_j(t.$$
In the rest of the present article, instead of writing $\gamma(0)\in\mathbb A_k^{d_1}\times_k\{0\}\times_k\{0\}$ we shall write for short $\gamma(0)\in\mathbb A_k^{d_1}$ for $\gamma$ in $\mathscr L(\mathbb A_k^d)$. We observe that the sets 
$$U_n:=\left\{\gamma\in \mathscr{X}_n(f)\mid \gamma(0)\in\mathbb A_k^{d_1}, \ord_tx(t)+\ord_ty(t)>n\right\}$$
and
$$W_n:=\left\{\gamma\in \mathscr{X}_n(f)\mid \gamma(0)\in\mathbb A_k^{d_1}, \ord_tx(t)+\ord_ty(t)\leq n\right\}$$
are closed and open $k$-subvarieties of $\mathscr X_n(f)\times_{X_0}\mathbb A_k^{d_1}$, respectively, and they are invariant under the natural $\mu_n$-action $\lambda\cdot \gamma(t)=\gamma(\lambda t)$. Hence we get the following identity in $\mathscr M_k^{\hat\mu}$:
\begin{align}\label{eq6.2}
\int_{\mathbb A_k^{d_1}}i^*\left[\mathscr{X}_n(f)\right]\L^{-nd}=\left[U_n\right]\L^{-nd}+\left[W_n\right]\L^{-nd}.
\end{align}
By Theorem \ref{Thm3.4}, the series 
$$U(T):=\sum_{n\geq 1}\left[U_n\right]\L^{-nd}T^n \ \text{ and } \ W(T):=\sum_{n\geq 1}\left[W_n\right]\L^{-nd}T^n$$ are rational series. Denoting $\mathbb{U}:=-\lim_{T\to \infty} U(T)$ and $\mathbb{W}:=-\lim_{T\to \infty} W(T)$ we obtain
\begin{align}\label{eq6.2}
\int_{\mathbb A_k^{d_1}}i^*\mathscr{S}_f=\mathbb{U}+\mathbb{W}.
\end{align}
The theorem hence follows directly from the following computations of $\mathbb{U}$ and $\mathbb{W}$.

\subsection{Computation of $\mathbb{U}$}\label{SS4.1}
In this paragraph we are going to prove the following proposition.
\begin{proposition}\label{ThmB}
The identity $\mathbb{U}=\L^{d_1}\mathscr S_{\tilde{f},0}$ holds in $\mathscr M_k^{\hat\mu}$.
\end{proposition}
\begin{proof}
Note that $f$ is of the form 
$$f(x,y,z)=\sum_{|\alpha|=|\beta|>0} c_{\alpha,\beta,\delta} x^{\alpha}y^{\beta}z^{\delta}+\tilde{f}(z),$$
where $|\alpha|=\alpha_1+\cdots+\alpha_{d_1}$ and $|\beta|=\beta_1+\cdots+\beta_{d_2}$. It implies that
\begin{align*}
U_n&=\left\{\gamma\in \mathscr{L}_n(\mathbb A_k^d)\mid \gamma(0)\in\mathbb A_k^{d_1}, \ord_tx(t)+\ord_ty(t)>n, f(\gamma)=t^{n}\mod t^{n+1}\right\}\\
&=\left\{\gamma\in \mathscr{L}_n(\mathbb A_k^d)\mid \gamma(0)\in\mathbb A_k^{d_1}, \ord_tx(t)+\ord_ty(t)>n, \tilde{f}(z)=t^{n} \mod t^{n+1}\right\}\\
&=U'_n\times \mathscr{X}_{n,0}(\tilde{f}),
\end{align*}
where 
$$U'_n:=\left\{(x,y)\in \mathscr{L}_n(\mathbb A_k^{d_1}\times \mathbb A_k^{d_2})\mid \ord_ty(t)>0, \ord_tx(t)+\ord_ty(t)>n\right\}.$$ 
Denote by $I$ the index set $\{1,2,\ldots,n,\infty\}$, and for each $m$ in $I$, put
\begin{align*}
U'_{n,m}:&=\left\{(x,y)\in U'_n\mid \ord_ty(t)=m\right\}\\
&=\left\{x\in \mathscr{L}_n(\mathbb A_k^{d_1})\mid \ord_tx(t)>n-m\right\}\times \left\{y\in \mathscr{L}_n(\mathbb A_k^{d_2})\mid \ord_ty(t)=m\right\}
\end{align*}
Then we have 
$$U'_n=\bigsqcup_{m\in I} U'_{n,m},$$ 
and, for every $1\leq m \leq n$,
$$[U'_{n,m}]=\L^{md_1}\cdot (\L^{d_2}-1)\L^{(n-m)d_2},$$
while for $m=\infty$,
$$U'_{n,\infty}=\L^{(n+1)d_1}.$$
It yields the identity 
\begin{align*}
[U'_{n}]=\sum_{m\in I} [U'_{n,m}]=\sum_{m=1}^n \L^{md_1}\cdot (\L^{d_2}-1)\L^{(n-m)d_2}+\L^{(n+1)d_1}
\end{align*}
in $\mathscr M_k^{\hat\mu}$, which implies that 
\begin{align}\label{eq6.3}
\sum_{n\geq 1}\left[U'_n\right]\L^{-n(d_1+d_2)}T^n=(\L^{d_2}-1)\sum_{1\leq m\leq n}\L^{-\left(nd_1+(d_2-d_1)m\right)}T^n+\L^{d_1}\sum_{n\geq 1}\L^{-nd_2}T^n .
\end{align}
Applying \cite[Lemma 2.10]{GLM1} we get
$$\lim_{T\to\infty}\sum_{1\leq m\leq n}\L^{-\left(nd_1+(d_2-d_1)m\right)}T^n=0,$$
hence
$$-\lim_{T\to\infty}\sum_{n\geq 1}\left[U'_n\right]\L^{-n(d_1+d_2)}T^n=\L^{d_1}.$$ 
We then conclude that
\begin{align*}
\mathbb{U}&=-\lim_{T\to\infty}\Big(\Big(\sum_{n\geq 1}\left[U'_n\right]\L^{-n(d_1+d_2)}T^n\Big) \ast Z_{\tilde{f},0}(T)\Big)\\
&=\Big(-\lim_{T\to\infty}\sum_{n\geq 1}\left[U'_n\right]\L^{-n(d_1+d_2)}T^n\Big) \cdot \Big(-\lim_{T\to\infty} Z_{\tilde{f},0}(T)\Big)\\
&=\L^{d_1}\mathscr S_{\tilde{f},0}.
\end{align*}
Here, the symbol $\ast$ stands for the Hadamard product of two series in $\mathscr M_k^{\hat\mu}[[T]]$ defined in \eqref{eq3.1}, and the second equality follows from Lemma \ref{Lem2}.
\end{proof}

\subsection{Computation of $\mathbb{W}$}\label{SS4.2}
We are now in position to prove the last proposition of the article and finish the proof of  the integral identity conjecture for regular functions.

\begin{proposition}
The identity $\mathbb{W}=0$ holds in $\mathscr M_k^{\hat\mu}$.
\end{proposition}


\begin{proof}
For $n\geq m\geq 1$, we define 
\begin{align}\label{Wnm}
W_{n,m}:=\left\{\gamma\in W_n\mid \ord_tx+\ord_ty=m\right\},
\end{align}
which is invariant under the natural $\hat\mu$-action $\lambda\cdot \gamma(t)=\gamma(\lambda t)$ (we recall that $\gamma=(x,y,z)$). It follows that 
\begin{align}\label{PT}
W_n=\coprod_{1\leq m\leq n}W_{n,m},
\end{align}
therefore,
\begin{align}\label{QTU}
W(T)=\sum_{1\leq m\leq n}\left[W_{n,m}\right]\L^{-nd}T^n,
\end{align}
a formal power series in $\mathscr M_k^{\hat\mu}[[T]]$. 

Let us consider the action of $\mathbb G_{m,k}$ on the affine $k$-variety $X:=\mathbb A_k^{d_1}\times_k \mathbb A_k^{d_2}\times_k \mathbb A_k^{d_3}$ given by 
$$\lambda\cdot (x,y,z):=(\lambda x,\lambda^{-1} y,z),$$ 
for $\lambda$ in $\mathbb G_{m,k}$ and $(x,y,z)$ in $X$. Since the group $\mathbb G_{m,k}$ is reductive and $X$ is affine, a categorical quotient $\phi: X\to Y$ exists and the space $Y$ has a structure of an algebraic $k$-variety (cf. \cite[Chapter 1, Section 2]{MFK94}). We consider the induced morphism 
$$\phi_n\colon \mathscr{L}_n(X)\to \mathscr{L}_n(Y)$$ 
and its restriction to $W_{n,m}$, also denoted by $\phi_n$,
$$\phi_n\colon W_{n,m}\to  V_{n,m}.$$
Here, $V_{n,m}$ denotes the image of $W_{n,m}$ in $\mathscr{L}_n(Y)$ under $\phi_n$. Since $W_{n,m}$ is invariant under the action of $\mu_n$, so is $V_{n,m}$, and furthermore the morphism $\phi_n$ is $\mu_n$-equivariant. We are going to show that $\phi_n$ is in fact a $\mu_n$-equivariant piecewise trivial fibration with fiber 
$$F:=\{\tau \in \mathscr{L}_{n}(\mathbb{A}^1_k)\mid \ord_t\tau<m\},$$
where the action of $\mu_n$ on $F$ is induced from the natural action of $\mu_n$ on $\mathscr{L}_{n}(\mathbb{A}^1_k)$. Indeed, for every field extension $K\supseteq k$, let us take an arbitrary $K$-arc $\psi$ in $V_{n,m}$ and consider the fiber $W_{n,m,\psi}$ of $\phi_n$ over $\psi$.   
Note that, for all $1\leq i \leq d_1$, $1\leq j\leq d_2$ and $1\leq l\leq d_3$, the morphisms $f_{ij}$ and $z_l$ from $X$ to $\mathbb{A}_k^1$ defined respectively by $x_iy_j$ and $z_l$ are $\mathbb{G}_{m,k}$-equivariant with respect to the trivial action of $\mathbb{G}_{m,k}$ on $\mathbb{A}_k^1$. By the universality of $\phi$, these morphisms are constant on every fiber of $\phi$. Hence, the induced morphisms $(f_{ij})_{n}, (z_l)_{n}\colon \mathscr{L}_{n}(X)\to \mathscr{L}_{n}(\mathbb{A}_k^1)$ are constant on $W_{n,m,\psi}$. It implies that, for any two elements $\gamma=(x,y,z)$ and $\gamma'=(x',y',z')$ in $W_{n,m,\psi}$, the following identities hold in $K[[t]]/(t^{n+1})$:
$$
\begin{array}{cll}
x_iy_j=x'_iy'_j, &\text{for} & 1\leq i \leq d_1, \ 1\leq j\leq d_2,\\
z_i=z'_i, &\text{for}  & 1\leq i \leq d_3.
\end{array}
$$

Fix an element $\gamma^{\circ}=(x^{\circ},y^{\circ},z^{\circ})$ in $W_{n,m,\psi}$. We may assume that $\ord_tx^{\circ}=\ord_tx_1^{\circ}$, and under this assumption, may prove that $\ord_tx=\ord_tx_1$ for all $\gamma=(x,y,z)$ in $W_{n,m,\psi}$. 
Let us define a morphism
$$\chi_{\psi}\colon W_{n,m,\psi} \to F\times_k \Spec K$$
which sends a $K$-arc $\gamma=(x,y,z)$ to its first component $x_1$. 
It is easy to see that $\chi_{\psi}$ is a $(\mu_n)_{\psi}$-equivariant morphism. Let us now prove that $\chi_{\psi}$ is isomorphic. For an arbitrary $v$ in $F\times_k \Spec K$, put $\tau=v(x_1^{\circ})^{-1} \in K((t))$ and 
$$(x,y,z)=(\tau x^{\circ} \mod t^{n+1}, \tau^{-1} y^{\circ}\mod t^{n+1}, z^{\circ}).$$ 
Then $(x,y,z)$ is in $W_{n,m,\psi}$ and $\chi_{\psi}(x,y,z)=v$, which proves that $\chi_{\psi}$ is a surjection. The injectivity of $\chi_{\psi}$ follows from the fact that $(f_{ij})_{n}$ and $(z_l)_{n}$ are constant on $W_{n,m,\psi}$. Therefore, the morphism $\chi_{\psi}$ is a $(\mu_n)_{\psi}$-equivariant isomorphism. By Theorem \ref{fund2}, $\phi_n$ is $\mu_{n}$-equivariant piecewise trivial fibration with fiber $F$, and we have
$$[W_{n,m}]=[V_{n,m}]\cdot [F]=[V_{n,m}]\cdot (\mathbb{L}^{n+1}-\mathbb{L}^{n-m+1})$$
in $\mathscr M_k^{\hat\mu}$. We consider the induced morphism of $\phi$ at the level of arc spaces 
$$\phi_{\infty}\colon \mathscr{L}(X)\to \mathscr{L}(Y)$$
and define a semi-algebraic family $A_m$ of semi-algebraic subsets of $\mathscr L(Y)$ as the image of the family $\{\gamma\in \mathscr{L}(X)\mid \ord_tx+\ord_ty=m\}$ under $\phi_{\infty}$. On the other hand, by the hypothesis, the regular function $f$ is $\mathbb G_{m,k}$-equivariant, it thus induces a regular function $g\colon Y\to \mathbb{A}^1_k$ satisfying $f=g\circ \phi$, by the universal property of the quotient $\phi$. In view of Theorem \ref{Thm3.4}, we define
$$A_{n,m}:=\left\{\gamma\in A_{m} \mid g(\gamma)=t^n\mod t^{n+1}\right\}.$$
Then $A_{n,m}$ is in $\mathscr F_Y^{\mu_n}$ and stable at level $n$, and 
$$\tilde{\mu}(A_{n,m})=\left[\pi_n\left(A_{n,m}\right)\right]\L^{-(n+1)(d-1)}=\left[V_{n,m}\right]\L^{-(n+1)(d-1)}$$
(note that $\dim_kY=d-1$). Therefore, we have the decomposition of $W(T)$ into the difference of two series as follows 
\begin{align*}
W(T)&=\sum_{1\leq m\leq n} \left[W_{n,m}\right]\L^{-nd}T^n\\
&=\L^d \sum_{1\leq m\leq n}\left[V_{n,m}\right]\L^{-(n+1)(d-1)}T^n-\L^d \sum_{1\leq m\leq n}\left[V_{n,m}\right]\L^{-(n+1)(d-1)}\L^{-m}T^n\\
&=\L^d\sum_{1\leq m\leq n}\tilde{\mu}\left(A_{n,m}\right)T^n-\L^d\sum_{1\leq m\leq n}\tilde{\mu}\left(A_{n,m}\right)\L^{-m}T^n.
\end{align*}
This implies that $\mathbb W=0$ in $\mathscr M_k^{\hat\mu}$, because the two series in the previous difference decomposition have the same limit, according to Proposition \ref{Prop4.5}.
\end{proof}


\begin{ack}
This article was partially written during the authors' visits to Department of Mathematics - KU Leuven in December 2017 and Vietnam Institute for Advanced Studies in Mathematics in January 2018. The authors thank sincerely these institutions for excellent atmospheres and warm hospitalities.
\end{ack}

\end{document}